\documentclass[11pt]{article}
\usepackage{amsmath}
\usepackage{amsfonts,amsthm,amssymb}
\usepackage{amstext}
\usepackage{graphicx}
\usepackage{subcaption}
\usepackage{geometry}
\usepackage{enumitem}
\usepackage{color}
\usepackage{algorithm}
\usepackage{algpseudocode}
\usepackage{multirow}
\usepackage{url}
\usepackage{threeparttable}
\usepackage{bm}
\usepackage{authblk} % Author affiliation

\newtheorem{proposition}{Proposition}

\newtheorem{lemma}{Lemma}
\newtheorem{corollary}{Corollary}
\newtheorem{theorem}{Theorem}
\newtheorem{definition}{Definition}
\newtheorem{remark}{Remark}

\newtheorem{example}{Example}

\newcommand{\conv}[1]{\mathrm{conv}(#1)}

\renewcommand{\Re}{\mathbb{R}}
\newcommand{\Ze}{\mathbb{Z}}
\newcommand{\Ne}{\mathbb{N}}

\usepackage{hyperref}
\usepackage[toc,page]{appendix}
\usepackage{framed}
\usepackage{tikz}

\tikzstyle{vertex}=[circle,minimum size=0.005\textwidth,draw,fill=white!100]

\newcommand{\prob}[1]{$(\mathrm{P})#1$}
\newcommand{\ext}[1]{\mathrm{ext}(#1)}

\newcommand{\polyv}{oligo-vertex} 
\newcommand{\inte}[1]{\mathrm{int}(#1)} % interior of a set
\newcommand{\bd}[1]{\partial #1} % boundary of a set

\title{\LARGE \bf
Optimal switching sequence for switched linear systems 
}

\author{Zeyang Wu\thanks{Email address: wuxx1164@umn.edu.} }
\author{Qie He\thanks{Email address: qiehe01@gmail.edu. Corresponding author.}}
\affil{{\small Department of Industrial \& Systems Engineering, University of Minnesota, USA}}
\date{\today}

\begin{document}

\maketitle

\begin{abstract}
We study the following optimization problem over a dynamical system that consists of several linear subsystems: 
Given a finite set of $n\times n$ matrices and an $n$-dimensional vector, find a sequence of $K$ matrices, each chosen from the given set of matrices, to maximize a convex function over the product of the $K$ matrices and the given vector.
This simple problem has many applications in operations research and control, yet a moderate-sized instance is challenging to solve to optimality for state-of-the-art optimization software.
We propose a simple exact algorithm for this problem.
Our algorithm runs in polynomial time when the given set of matrices has \emph{the \polyv{} property}, a concept we introduce in this paper for a finite set of matrices.
We derive several sufficient conditions for a set of matrices to have the \polyv{} property.
Numerical results demonstrate the clear advantage of our algorithm in solving large-sized instances of the problem over one state-of-the-art global optimization solver.
We also propose several open questions on the \polyv{} property and discuss its potential connection with the finiteness property of a set of matrices, which may be of independent interest.  

%\keywords{Switched linear systems \and Optimal control \and Discrete optimization \and Computational complexity \and Convex geometry \and Enumerative combinatorics}
 %\subclass{05A16 \and 90C27 \and 68Q25 \and 93C30 \and 37N40}
\end{abstract}

%%%%%%%%%%%%%%%%%%%%%%%%%%%%%%%%%%%%%%%%%%%%%%%%%%%%%%%%%%%%%%%%%%%%%%%%%%%%%%
%%Introduction
%%%%%%%%%%%%%%%%%%%%%%%%%%%%%%%%%%%%%%%%%%%%%%%%%%%%%%%%%%%%%%%%%%%%%%%%%%%%%%
\section{Introduction} \label{sec:intro}
Many real-world systems exhibit significantly different dynamics under various modes or conditions, for example a manual transmission car operating at different gears, a chemical reactor under different temperatures and flow rates of reactants, and a group of cancer cells responding to different drugs.
Such phenomena can be modeled under a unified framework of switched systems.
A switched system is a dynamical system that consists of several subsystems and a rule that specifies the switching among the subsystems.
Finding a switching rule to optimize the dynamics of a switched system under certain criteria has found numerous applications in power system operations, chemical process control, air traffic
management, and medical treatment design~\cite{sun2006switched,lin2009stability,liberzon2012switching,he2016optimized}.
%
%
%We are interested in an optimal control problem in switched linear systems.
%
In this paper, we study the following discrete-time switched linear system: %It can be described by the following difference equations:
\begin{equation} \label{eq_switched_system}
x(k+1) = T_k x(k),  \qquad T_k \in \Sigma, \; k=0,1,\ldots,
\end{equation}
where $x(k)$ is an $n$-dimensional real vector that captures the system state at period $k$, the set $\Sigma$ contains $m$ real matrices in $\Re^{n \times n}$, each of which describes the dynamics of a linear subsystem, and the initial vector $x(0)$ is a given $n$-dimensional real vector $a$.
Such a system with switching only at fixed time instants appear in many practical applications, and is also employed to approximate the more complex dynamics of a continuous-time hybrid system with switching times defined over the real line~\cite{sun2006switched,liberzon2012switching}.

We are interested in the following optimization problem \prob{} related to the system in~\eqref{eq_switched_system}:
\begin{framed}
%Given a pair of matrices $A, B \in \Qe^{n \times n}$, an initial state $x(0) \in \Qe^n$, a cost vector $c \in \Qe^n$, and a positive integer $K$, find a sequence of matrices $T_0, T_1, \ldots, T_{K-1} \in \{A,B\}$ to minimize $c^{\top}x(K)$, where $x(k + 1) =T_k x(k), \; k = 0, \dots,K-1$.
Given a switched linear system described in~\eqref{eq_switched_system}, a positive integer $K$, and a convex function $f:\Re^n \rightarrow \Re$, find a sequence of $K$ matrices $T_0, T_1, \ldots, T_{K-1} \in \Sigma$ to maximize $f(x(K))$.
\end{framed}

\noindent One type of such convex functions are the $\ell_p$ norms.

\begin{example}
Consider a switched linear system consisting of two subsystems with system matrices  
$A  = 
\begin{bmatrix}
1 & 1 \\
1 & 0
\end{bmatrix}$ and
$B = \begin{bmatrix}
1 & 1 \\
0 & 1
\end{bmatrix}$, an initial vector $a=(2,1)^{\top}$, and $K=8$.
Figure~\ref{fig:trajectories} illustrates the trajectory of $x(k)$ under three different switching sequences,
%: $A-B-B-A-B-B-A-B$, $B-A-B-A-B-A-B-A$, $A-A-B-A-A-B-A-A$, respectively.
%
with the final state $x(8)$ being $(53, 23)^{\top}$, $(58, 41)^{\top}$, and $(71, 41)^{\top}$, respectively.
\begin{figure}[ht]
\centering
\includegraphics[scale=0.5]{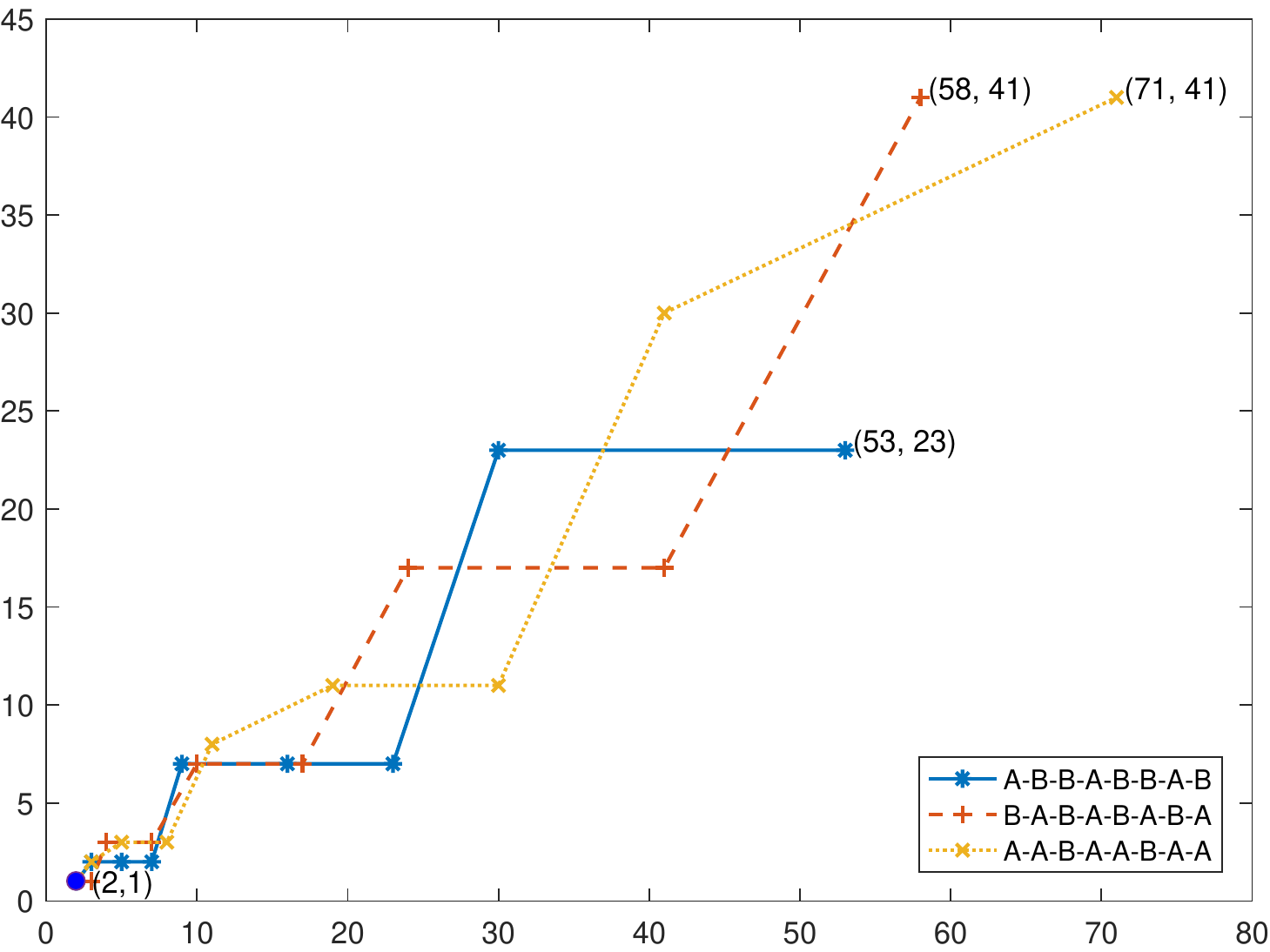}
\caption{The trajectory of $x(k)$ under three matrix sequences}
\label{fig:trajectories}
\end{figure}
\end{example}

\noindent We give three examples below to illustrate the applications of Problem \prob{} and its connection to other problems in control and optimization.

The first example is on design of treatment plans.
Antibiotic resistance renders diseases that were once easily treatable dangerous infections, and has become one of the most pressing public health problems around the world.
Several groups of researchers studied how to design sequential antibiotic treatment plans to restore susceptibility after bacteria develop resistance~\cite{mira2015rational, nichol2015steering}.
%
%They proposed a problem called the antibiotics time machine problem.
%
%In particular, they modeled the evolution between various genotypes of TEM $\beta$-lactamase (an enzyme produced by bacteria) under each antibiotic by a probability transition matrix, where the transition probabilities (the mutation rates between different genotypes) are determined by the fitness values of all the genotypes under the particular antibiotic.
%
%The goal is to find a sequence of antibiotics to maximize the percentage of wild type TEM $\beta$-lactamase.
%, which is susceptible to all $\beta$-lactam antibiotics.
%
%
%This problem can be modeled as \prob{} in the following way.
%
%Suppose the enzyme has $n$ genotypes. 
%
They model the percentages of $n$ genotypes of an enzyme produced by bacteria in a population after $k$ periods of treatment with the vector $x(k)$, and model the mutation rates among $n$ genotypes under each antibiotic with a probability transition matrix.
The goal is to design a sequence of antibiotics to maximize the percentage of the wild type at the end of the treatment, which is sensitive to all antibiotics.
The treatment design problem is equivalent to solve \prob{} with $a=e_1$, a unit vector with the first component being 1 which denotes $100\%$ wild type in the beginning, and $f(x(K)) = -e_1^\top x(K)$.
In the same vein, \prob{} can model the sequential therapy design problem for many other diseases when $x(k)$ describes the related biometrics of a patient at period $k$ and each matrix models the evolution of patient biometrics under a particular treatment~\cite{he2016optimized}.

The second example is the matrix mortality problem in control~\cite{blondel1997pair,bournez2002mortality}.
Given a positive integer $k$, a set of matrices is said to be $k$-mortal if the zero matrix can be expressed as a product of $k$ matrices in the set (duplication allowed).
A set of matrices is said to be mortal if it is $k$-mortal for some finite $k$.
The matrix mortality problem captures the stability of switched linear systems under certain switching rules.
%
%Roughly speaking, if a set of matrices is mortal, then there exists a switching rule that is able to stabilize the switched linear system defined by the set of matrices.
%
%It also belongs to a broad family of problems in algebra called the membership problem, which asks whether a target element is a member of an algebraic structure (group, semigroup, monoid, etc.) generated by a finite set of given elements~\cite{cai2000complexity,bell2007computational}.
%
%For the matrix mortality problem, the target element is the zero matrix and the algebraic structure is the semigroup generated by the given set of matrices under matrix multiplication.
%
%The matrix mortality problem and matrix $k$-mortality problem for specific families of matrices have been studied extensively in the literature~\cite{blondel1997pair,bournez2002mortality}.
%
%In particular, testing whether a pair of $n \times n$ matrices is $k$-mortal is known to be NP-complete, even for a pair of binary matrices\cite{blondel1997pair}.
%
It can be shown that a finite set of $n \times n$ non-negative matrices is $k$-mortal if and only if the optimal objective value of \prob{} is 0 with $a = \mathbf{1}$, $K=k$, and $f(x(K)) = -\mathbf{1}^\top x(K)$, where $\mathbf{1}$ is an $n$-dimensional vector with each component being 1.

The third example concerns the joint spectral radius of a set of matrices, an important quantity which has found many applications in wavelet functions, constrained coding, and network security management, etc~\cite{jungers2009joint}. The joint spectral radius of a finite set $\Sigma$ of matrices is defined as $\rho(\Sigma) = \limsup_{k \rightarrow \infty} \hat{\rho}_k(\Sigma, \|\cdot\|)$~\cite{rota1960note},
where
\begin{equation} \label{eq:rhok}
\hat{\rho}_k(\Sigma, \|\cdot\|)=\max\{\|T_{k-1}T_{k-2}\ldots T_0\|^{1/k} \mid T_j \in \Sigma, j=0,\ldots,k-1\}
\end{equation}
and $\|\cdot\|$ is some matrix norm.
%
%Note that $\rho(\Sigma)$ does not depend on the chosen norm due to the equivalence of norms in $\Re^n$.
%
%The joint spectral radius, first introduced by Rota and Strang in 1960~\cite{rota1960note}, characterizes the maximum asymptotic growth of a vector under a sequence of linear mappings selected from a given set.
%
%It generalizes the spectral radius of one matrix to a set of matrices.
%
%The joint spectral radius has found wide applications in a variety of seemingly irrelevant fields, such as wavelet functions, switching systems, constrained coding, and network security management~\cite{jungers2009joint}.
%
%Computing the joint spectral radius of a given set of matrices is in general a difficult task, and is still a topic of active research.
%
If we select the matrix norm in~\eqref{eq:rhok} to be induced by the $\ell_p$ norm of a vector, then
\begin{equation} \label{eq:rhok:lp}
(\hat{\rho}_k(\Sigma, \|\cdot\|))^k
%\max\{\|T_{k-1}T_{k-2}\ldots T_0\|^{1/k} \mid T_j \in \Sigma, j=0,\ldots,k-1\}
= \sup_{\|a\|_p=1}\max\{\|x(K)\|_p \mid  \eqref{eq_switched_system}\}.
\end{equation}
Observe that the inner maximization problem of the right-hand side of~\eqref{eq:rhok:lp} is a special case of \prob{} with the convex function $f(x) = \|x\|_p$.
In general, let $v^*$ be the optimal objective value of \prob{} with $f(x)=\|x\|$ for some norm $\|\cdot\|$ and an initial vector $a$. Then $(v^*)^{1/k}$ provides a lower bound of the quantity $\hat{\rho}_k(\Sigma, \|\cdot\|)$.
%

%Our focus of this paper is to study the computational complexity of \prob{}.

%
%
A simple way to solve \prob{} is to enumerate all possible matrix sequences, but such an approach quickly becomes impractical as $m$ and $K$ increase.
Even for $m=5$ and $K=30$, we need to enumerate $5^{30}$ solutions, a formidable task for the current fastest computer. 
Another general approach to solve \prob{} is to formulate it as a mixed-integer nonlinear optimization problem, which can be solved by global optimization solvers, but the problem size that can be handled by state-of-the-art commercial solvers is also limited.
%
%See Section~\ref{sec:computation} for more detailed computational results. 
%
%
In addition, the time complexity of the tree-based search algorithms employed by these global solvers is difficult to analyze in general. 
In many applications, problem \prob{} has to be solved repeatedly with different parameters, so it is of vital importance to have a fast algorithm for \prob{}. 
%
%\red{When $f$ is linear, \prob{} can be formulated as a mixed-integer linear program and solved by state-of-the-art commercial optimization software, such as Gurobi. Some preliminary numerical results show that Gurobi has difficulty solving \prob{} with moderate-size $n$ ($n=?$) and large $K$ ($K=?$).}

%
We now present our results.
We develop a simple dynamic programming algorithm to solve \prob{} exactly, which means that an optimal matrix sequence is guaranteed at the termination of the algorithm.
Our algorithm is much faster than the state-of-the-art global optimization solver Baron in solving the same instance of \prob{}.
%, based on computational results in Section~\ref{sec:computation}.
%
%
Another advantage of our algorithm is that it does not require any additional property of the function $f$ such as smoothness or strong convexity.
The main idea of our algorithm is to find out the extreme points of a polytope by iteratively constructing the convex hull of linear transformations of another polytope's extreme points. As pointed out by one referee, this idea has been used before to construct a special polytope needed to compute the joint spectral radius of a finite set of matrices~\cite{guglielmi2013exact}.

Furthermore, we introduce a new concept for a finite set of matrices to analyze the time complexity of our algorithm. 
In particular, we assume that all input data are integers and the value of the convex function $f$ can be queried through an oracle in constant time; we adopt the random-access machine~\cite{papadimitriou2003} as the model of computation, in which each basic operation (addition, comparison, multiplication, etc.) is assume to take the same amount of time and the time complexity of an algorithm is the number of steps/operations required to execute the algorithm.
%
%
%In addition, we assume that the total number of periods $K$ is given in unary encoding.
%
We define the following notations that are useful for presenting the time-complexity results.
Given a finite set $\Sigma$ of $n \times n$ real matrices and a vector $a \in \Re^n$, let
\[P_k(\Sigma, a):= \conv{\{x(k) \mid x(k)= T_{k-1}\cdots T_0 a, T_j \in \Sigma, j=0,\ldots,k-1\}}\]
be the convex hull of all possible values of $x(k)$ in~\eqref{eq_switched_system} for each integer $k \ge 0$.
Let $N_k(\Sigma, a)$ be the number of extreme points of $P_k(\Sigma, a)$
and
\begin{equation*} %\label{eq:nk}
N_k(\Sigma) =\sup_{a \in \Re^n}\{N_k(\Sigma, a)\}.
\end{equation*}
%
%
%Since $P_k(\Sigma, a)$ is the convex hull of at most $2^k$ points, both $N_k(\Sigma, a)$ and $N_k(\Sigma)$ are well defined and bounded above by $2^k$.
%
%It is not difficult to see that the value of $N_k(\Sigma)$ is independent of the norm $\|\cdot\|$ used in~\eqref{eq:nk}.
%
We introduce the following concept for a set of matrices.
\begin{definition} \label{def:oligo-vertex}
A set of matrices $\Sigma$ is said to have \textbf{the \polyv{} property} if there exists $\alpha >0$, positive integer $k_0$, and positive constant $d$ such that $N_k(\Sigma) \le \alpha k^d$ for any $k \ge k_0$.
\end{definition}
\noindent The \polyv{} property of a set of matrices indicates the number of extreme points of $P_k(\Sigma, a)$ grows at most polynomially in $k$ for any initial vector $a$, despite the number of possible values of $x(k)$ grows exponentially with $k$ in general.
With the big-Oh notation commonly used in computer science, the \polyv{} property basically states that $N_k(\Sigma)=O(k^d)$ as $k \rightarrow \infty $ for some positive constant $d$.
%
%Meanwhile with the $\limsup$ notation, the \polyv{} property holds if and only if there exists a positive constant $d$ such that $\limsup\limits_{k\rightarrow \infty} \frac{N_k(\Sigma)}{k^d} < \infty$.

\subsection*{Our contributions}
We summarize the contributions of this paper as follows.
\begin{enumerate}
	\item We present a simple dynamic programming algorithm to solve \prob{} exactly. 
	Our algorithm does not require any additional property of $f$ other than convexity.
	Numerical experiments demonstrate that the algorithm is much faster than state-of-the-art global optimization software in solving large-sized instances.
	Our algorithm can be considered as a variant of the algorithm for computing the joint spectral radius in~\cite{guglielmi2013exact} with the same basic idea. On the other hand, as it is applied to a different problem, changes such as initialization, the pruning rule, and termination conditions have been made.
%	\sloppy The running time of our algorithm is $O(m^2n^{4.5}(\log n + \log M) \sum_{k=0}^{K-1}k N_k(\Sigma)^2)$, and can be reduced to $O(m \log m \sum_{k=0}^{K-1}N_k(\Sigma) + m\sum_{k=0}^{K-1}N_k(\Sigma) \log N_k(\Sigma))$ when $n=2$, where $M$ is the maximum absolute value of the entries of $A_1, \ldots, A_m$, and $a$. 
	\item We introduce the concept of the \polyv{} property for a finite set of matrices, and show that our algorithm runs in polynomial time if the given set of matrices has the \polyv{} property.
	To the best of our knowledge, this is the first time such a property is introduced for a set of matrices.	
	\item
	We derive several sufficient conditions for a set of matrices to have the \polyv{} property.
	%
	%
	%In particular, we show that a pair of $2 \times 2$ binary matrices has the \polyv{} property.
	%
	On the other hand, we show that \prob{} is NP-hard for a pair of stochastic matrices or a pair of binary matrices, which implies that the \polyv{} property is unlikely to hold for an arbitrary pair of $n \times n$ matrices unless P=NP.
	Finally we propose several open questions on the \polyv{} property.
	%	
	%\item We show that the oligo-vertex property holds for a pair of commuting matrices, a pair of projection matrices, a pair of matrices with at least one singular matrix, a pair of $2 \times 2$ matrices that share exactly one eigenvectors, and the most nontrivial case, a pair of $2 \times 2$ binary matrices.
\end{enumerate}
%

%\noindent Binary matrices are important classes of matrices for studying matrix $k$-mortality problem~\cite{blondel1997pair} and computing joint spectral radius~\cite{moision2001codes,jungers2008,jungers2009joint}.
%%
%Since $x(k)$ can take at most $2^k$ different values for a pair of matrices, one would expect that the number of extreme points of the convex hull of all values of $x(k)$ may also grow exponentially in $k$.
%%
%But as we show later, this growth rate for a pair of $2 \times 2$ binary matrices is very small, mostly linear or quadratic in $k$.
%%
%On the other hand, the NP-hardness of \prob{} with $n$ being an input parameter indicates that unless P=NP, the \polyv{} property is unlikely to hold for an arbitrary pair of integer matrices.
%
%
%
%
%
%Some techniques developed in this paper for proving a pair of matrices in $\Re^2$ to have the \polyv{} property can be applied to matrices in arbitrary dimensions.

\noindent The \polyv{} property we propose may be of independent interest to readers. 
We want to point out some similarities between the \polyv{} property and another important property for a set of matrices that is also concerned with long matrix products---the finiteness property.
A finite set $\Sigma$ of matrices is said to have the finiteness property if the joint spectral radius $\rho(\Sigma)$ is equal to $(\rho_k(T_{k-1}T_{k-2}\ldots T_0))^{1/k}$ with $T_{k-1}, T_{k-2}, \ldots, T_0 \in \Sigma$ for some finite integer $k$, where $\rho(T)$ denotes the spectral radius of the matrix $T$.
The finiteness property has been studied extensively for different families of matrices~\cite{lagarias1995finiteness, jungers2009finiteness}, as it has many implications on stability and stabilization of switched systems.
The finiteness property and the \polyv{} property both hold for the following sets of matrices: commuting matrices, any finite set of matrices with at most one matrix's rank being greater than one~\cite{liu2013rank}, and a pair of $2\times 2$ binary matrices~\cite{jungers2008}.
%
%
%%
%Inspired by the challenge of studying the finiteness property, we intend that our result as a first step to study the \polyv{} property of a set of matrices.
%
We suspect that there is a deeper connection between these two properties.
We pose several open questions on the \polyv{} property at the end of this paper.
%To the best of our knowledge, there is no unified argument to show that any pair of $2 \times 2$ binary matrices has the finiteness property.
%%
%This is similar to our result for proving the \polyv{} property of $2 \times 2$  binary matrices.
%%
%Both illustrate the challenges of proving this type of results for a set of matrices.
%

%
%Finally we have the following conjecture.
%% based on some preliminary computational results.
%%
%\begin{conjecture} \label{conj:complexity:2by2matrices}
%Any pair of $2 \times 2$ real matrices has the \polyv{} property, and the corresponding \prob{} is polynomially solvable for a pair of $2 \times 2$ integer matrices.
%\end{conjecture}

The rest of the paper is organized as follows.
In Section~\ref{sec:review}, we review results related to the problem we study, with a main focus on computational complexity.
In Section~\ref{sec:complexity}, we first prove that \prob{} is NP-hard for a pair of stochastic matrices or binary matrices, and then introduce an exact algorithm for \prob{} and analyze its time complexity for general $n$ and $n=2$.
%
%In Section~\ref{sec:algorithm}, we introduce an algorithm to solve \prob{} and a nalyze its complexity in $\Re^n$ for general $n$ and $n=2$.
%
In Section~\ref{sec:polytime}, we present several sufficient conditions for a set of matrices to have the \polyv{} property.
In Section~\ref{sec:binary}, we prove that a pair of $2 \times 2$ binary matrices has the \polyv{} property.
We present some computational results in Section~\ref{sec:computation}, and conclude in Section~\ref{sec:conclusion} with some open problems.

\section{Related Work}
\label{sec:review}
%%%%%%%%%%%%%%%%%%%%%%%%%%%%%%%%%%%%%%%%%%%%%%%%%%%%%%%%%%%%%%%%%%%%%%%%%%%%%%
%%Literature Review
%%%%%%%%%%%%%%%%%%%%%%%%%%%%%%%%%%%%%%%%%%%%%%%%%%%%%%%%%%%%%%%%%%%%%%%%%%%%%%
Our problem aims to find the optimal switching rule of a discrete-time switched linear system without continuous control input.
There have been a rich body of theoretical and computational results on optimal control of switched linear systems, such as finding optimal switching instants given a fixed switching sequence~\cite{yuan2015hybrid}, minimizing the number of switches with known initial and final states~\cite{egerstedt2003optimal}, finding suboptimal policies~\cite{antunes2017linear}, study of the exponential growth rates of the trajectories under different switching rules~\cite{hu2011generating}, and characterizing the value function of switched linear systems with linear and quadratic objectives~\cite{zhang2009value}.
We refer interested readers to the books~\cite{sun2006switched,liberzon2012switching} and recent surveys~\cite{sun2005analysis,zhu2015optimal} for more details on switched linear systems.
Finding the optimal switching sequence for a switched linear system also belongs to a broader class of problems called mixed-integer optimal control~\cite{sager2005numerical,sager2012integer} or optimal control of hybrid systems~\cite{antsaklis2000brief}, which can be formulated as a mixed-integer nonlinear optimization problem and solved by general mixed-integer optimization solvers.
%
%The reformulated problem can be solved by general tree-based algorithms such as branch-and-bound, but the running time of the software becomes very bad (*****) and the time complexity of these algorithms are difficult to analyze.
%and the performance of the algorithm usually deteriorates quickly as the problem size increases.
%As for algorithms for the more general mixed-integer optimal control problems, see~\cite{sager2005numerical,sager2009direct,sager2012integer,Kirches2016}. 
%
%The time complexity of the algorithm for the formulated mixed-integer nonlinear optimization problem, however, is in general difficult to analyze.
%

%We mainly focus on the computational complexity of the problem and aim to derive an algorithm that runs in polynomial time in the.
%Our result is mainly on deriving algorithms with theoretically.

We now survey results in the literature that are closely related to the problem we study.
Blondel and Tsitsiklis showed that the matrix mortality problem is undecidable for a pair of $48 \times 48$ integer matrices and the matrix $k$-mortality problem is NP-complete for a pair of $n\times n$ binary matrices with $n$ being an input parameter~\cite{blondel1997pair}.
The complexity of the matrix $k$-mortality problem is however unknown when the matrix dimension $n$ is fixed.
%
% Antibiotics time machine
For the antibiotics time machine problem, Mira et al. used exhaustive search to find the optimal sequence of antibiotics for a small sized problem~\cite{mira2015rational}.
Tran and Yang showed that the antibiotics time machine problem is NP-hard when the number of matrices and the matrix dimension are both input parameters~\cite{tran2015antibiotics}.
%
%We show that the problem is NP-hard even for a pair of stochastic matrices.
%
The antibiotics time machine can be also seen as a special finite-horizon discrete-time Markov decision process in which no state is observable.
It has been shown in~\cite{papadimitriou1987} that the finite-horizon unobservable Markov decision process is NP-hard.
Therefore, our results identify several polynomially solvable cases of finite-horizon unobservable Markov decision processes.
%
% Joint spectral radius
Computing the joint spectral radius for a finite set of matrices either exactly or approximately has been shown to be NP-hard~\cite{tsitsiklis1997lyapunov}, and
has been a topic of active research~\cite{blondel2005comp,parrilo2008app,ahmadi2014joint}.
Guglielmi and Protasov proposed an algorithm to compute the joint spectral radius of a finite set of matrices~\cite{guglielmi2013exact}. 
The key component of the algorithm is to construct a special polytope $P$ from which the joint spectral radius of the given set of matrices can be computed accordingly. 
%
%This polytope $P$ needs to satisfy the condition that the Minkowski functional of $P$ gives an extremal norm of the given set of matrices.
%
Similar to our algorithm, the polytope $P$ is constructed by finding out its extreme points, through an iterative procedure of taking the convex hull of linear transformation of extreme points of another polytope.
% iteratively constructs a sequence of polytopes, using linear programs to check if a given point is an extreme point of the convex hull of a set of points. 
%
However, the purposes, the running time, and the implementation details of the two algorithms are different.	
The algorithm in~\cite{guglielmi2013exact} aims to find a polytope that gives an extremal norm for the given set of matrices, and only terminates in finite time for the set of matrices satisfying certain conditions.
On the other hand, our algorithm aims to construct the convex hull of all possible states reachable by the switched system after $K$ periods, and will always terminate after exactly $K$ periods for any given set of matrices.
The algorithm in~\cite{guglielmi2013exact} has recently been improved in~\cite{mejstrik2018improved}.
The finiteness conjecture~\cite{lagarias1995finiteness}, which states that the finiteness property holds any set of real matrices, had remained a major open problem in the control community until early 2000s when a group of researchers showed that there exists a pair of $2 \times 2$ matrices that does not have the finiteness property~\cite{bousch2002asymptotic,blondel2003elementary,kozyakin2005dynamical}.
%%
%%
%The finiteness conjecture was shown to be false in early 2000s by non-constructive proofs~\cite{bousch2002asymptotic,blondel2003elementary,kozyakin2005dynamical}.
%
The first constructive counterexample for the finiteness conjecture was proposed in~\cite{hare2011explicit}.
The finiteness conjecture was shown to be true for a pair of $2 \times 2$ binary matrices~\cite{jungers2008} and a finite set of matrices with at most one matrix's rank being greater than one~\cite{liu2013rank}.
%
%The third example is related to Markov decision processes (MDPs). In \prob{}, if the initial vector $x(0)$ denotes a probability distribution, i.e., $x_i(0)\ge 0$ for $i\in [n]$ and $\sum_{i\in [n]}x_i(0)=1$, both matrices $A$ and $B$ are the transpose of some probability transition matrices, and $f(x)=c^{\top}x$, then \prob{} describes a particular finite-horizon discrete-time MDP. This MDP has $n$ states, but none of them is observable and we only know the initial distribution of the states is $x(0)$; at each time period, only two actions can be taken and the corresponding probability transition matrices are $A^{\top}$ and $B^{\top}$, respectively; there is a terminal reward $c_i$ for state $i \in [n]$ at time $K$. Then \prob{} tries to find a sequence of actions over $K$ time periods to maximize the expected terminal rewards.
%%

\section{Computational Complexity and the Algorithm}
\label{sec:complexity}
\subsection{Notations}
We first introduce some notations that will be used throughout this paper.
Let $\Ne$, $\Ze$, $\Re$, $\Re_{+}$, and $\Re_{-}$ denote the sets of natural numbers (including 0), integers, real numbers, non-negative real numbers, and non-positive real numbers, respectively.
We use $x_i$ to denote the $i$-th component of a given vector $x$.
%
%Let $\mathbf{0}$ denote the origin in $\Re^n$.
%Given $x \in \Re$, let $\lfloor x \rfloor$ and $\lceil x \rceil$ denote the maximum integer no greater than $x$ and the minimum integer that no less than $x$, respectively.
%
Let $\|x\|_\infty$ and $\|T\|_\infty$ denote the infinity norm of vector $x$ and matrix $T$, respectively.
Given two positive integers $i,j$, let $[i:j]$ denote the set of integers $\{i,i+1, \ldots,j\}$ if $i\le j$ and $\emptyset$ if $i>j$.
Given two scalar functions $f$ and $g$ defined on some subset of real numbers, we write $f(x)=O(g(x))$ as $x \rightarrow \infty$, if there exist $\alpha$ and $x_0 \in \Re$ such that $|f(x)| \le \alpha|g(x)|$ for all $x \ge x_0$.
Given a set $S$, let $|S|$ denote the cardinality of $S$, $\conv{S}$ denote the convex hull of $S$, $\inte{S}$ denote the interior of $S$, and $\bd{S}$ denote the boundary of $S$, respectively.
Let $\ext{S}$ denote the set of extreme points of a convex set $S$.
Given a set $S \subseteq \Re^n$ and a matrix $T \in \Re^{n\times n}$, let $TS:= \{Tx \mid x\in S\}$ be the image of $S$ under the linear mapping defined by $T$.
Let $\mathcal{Q}_i$ denote the $i$-th quadrant of the plane under the standard two-dimensional Cartesian system, for $i=1,2,3,4$.
For example, $\mathcal{Q}_1=\{x \in \Re^2 \mid x_1 \ge 0, x_2 \ge 0\}$.

\subsection{Complexity}
\begin{theorem}\label{thm:NPhard}
\prob{} is NP-hard for a pair of left (right) stochastic matrices and a linear function $f$.
\end{theorem}
%\noindent The proof is based on a reduction from the 3-SAT problem. We omit the proof due to the limitation of space.
\begin{proof}
We prove the result based on a reduction from the 3-SAT problem. A 3-SAT problem asks whether there exists a truth assignment of several variables such that a given set of clauses defined over these variables, each with three literals, can all be satisfied. The 3-SAT problem is known to be NP-complete~\cite{garey1979computers}. 

Given an instance of the 3-SAT problem with $n$ variables $y_1,\ldots,y_n$ and $m$ clauses $C_1, \ldots, C_m$, we construct an instance of \prob{} with $\Sigma=\{A, B\}$ as follows. Matrices $A$ and $B$ are $m(2n+1) \times m(2n+1)$ adjacency matrices of two directed graphs $G_A$ and $G_B$, respectively. The construction of $G_A$ and $G_B$ will be explained in detail below. We set the total number of periods $K=n$. Let $e_{k} \in \Re^{m(2n+1)}$ be a vector with the $k$-th entry being 1 and all other entries being 0. We set $x(0)=\sum_{j=1}^{m} e_{(j-1)(2n+1)+1}$ and $f(x)=c^{\top}x$ with $c=-\sum_{j=1}^{m}e_{j(2n+1)}$. We claim that the 3-SAT instance is satisfiable if and only if the optimal objective value of the constructed instance of \prob{} is $-m$.

Graph $G_A$ is constructed as follows. It contains $m(2n+1)$ nodes, divided equally into $m$ groups, each group corresponding to a clause. There is no arc between nodes in different groups. Let $u_{j,1}, u_{j,2}, \ldots, u_{j, 2n+1}$ be the $2n+1$ nodes corresponding to clause $j$. The arcs among these nodes are as follows. Node $u_{j,2n+1}$ has a self loop. There is an arc from $u_{j,l+1}$ to $u_{j,l}$ for $l=[1:2n]$ unless literal $y_l$ is included in clause $C_j$; in that case, there will be an arc from node $u_{j,n+l+1}$ to node $u_{j,l}$. Graph $G_B$ is constructed similarly with the same set of nodes. There is an arc from $u_{j,l+1}$ to $u_{j,l}$ for $l=[1:2n]$ unless literal $y^c_l$ is included in clause $C_j$; in that case, there will be an arc from node $u_{j,n+l+1}$ to node $u_{j,l}$.
%There are arcs $(u_{j,2n+1}, u_{j,2n})$, $(u_{j,2n}, u_{j,2n-1}), \ldots$, $(u_{j,n+2}, u_{j,n+1})$. There is an arc $(u_{j,l+1}, u_{j,l})$ for $l=[1:n]$ unless that literal $y_l$ is included in clause $C_j$. In that case, there will be an arc from node $u_{j,n+l+1}$ to node $u_{j,l}$. The graph $G_B$ is constructed similar to that of graph $G_A$, with $m$ groups of $2n+1$ nodes. Node $u_{j,2n+1}$ has a self loop and there are arcs $(u_{j,2n+1}, u_{j,2n})$, $(u_{j,2n}, u_{j,2n-1}), \ldots$, $(u_{j,n+2}, u_{j,n+1})$. The difference between $G_A$ and $G_B$ is that there is an arc $(u_{j,l+1}, u_{j,l})$ for $l=[1:n]$ unless that literal $y^c_l$ is included in clause $C_j$, in which case there will be an arc from node $u_{j,n+l+1}$ to node $u_{j,l}$.
An example for the clause $C_j=y_1\vee y^c_3 \vee y_4$ with a total of $4$ variables is shown in Figure~\ref{fig:stochasticmatrices}.
\begin{figure}[ht]
  \centering
\begin{tikzpicture}[scale=0.5]
    \node [vertex] (1) at (0,0) {$u_{j,1}$};
		\node [vertex] (2) at (0,-3) {$u_{j,2}$} ;
		\node [vertex] (3) at (0,-6) {$u_{j,3}$} ;
		\node [vertex] (4) at (0,-9) {$u_{j,4}$} ;
		\node [vertex] (5) at (4,0) {$u_{j,5}$} ;
		\node [vertex] (6) at (4,-3) {$u_{j,6}$} ;
		\node [vertex] (7) at (4,-6) {$u_{j,7}$} ;
		\node [vertex] (8) at (4,-9) {$u_{j,8}$} ;
		\node [vertex] (9) at (4,-12) {$u_{j,9}$} ;
		\node at (2,-15) {Part of $G_A$};
		%\draw[->] (2) -- (1);
		\draw[->] (6) -- (1);
		\draw[->] (3) -- (2);
		\draw[->] (4) -- (3);
		\draw[->] (9) -- (4);
		\draw[->] (6) -- (5);
		\draw[->] (7) -- (6);
		\draw[->] (8) -- (7);
		\draw[->] (9) -- (8);		
		\path[->] (9) edge [loop below] (9);
\end{tikzpicture}
\hspace{5mm}
\begin{tikzpicture}[scale=0.5]
    \node [vertex] (1) at (0,0) {$u_{j,1}$};
		\node [vertex] (2) at (0,-3) {$u_{j,2}$} ;
		\node [vertex] (3) at (0,-6) {$u_{j,3}$} ;
		\node [vertex] (4) at (0,-9) {$u_{j,4}$} ;
		\node [vertex] (5) at (4,0) {$u_{j,5}$} ;
		\node [vertex] (6) at (4,-3) {$u_{j,6}$} ;
		\node [vertex] (7) at (4,-6) {$u_{j,7}$} ;
		\node [vertex] (8) at (4,-9) {$u_{j,8}$} ;
		\node [vertex] (9) at (4,-12) {$u_{j,9}$} ;
		%\draw[->] (2) -- (1);
		\draw[->] (2) -- (1);
		\draw[->] (3) -- (2);
		\draw[->] (8) -- (3);
		\draw[->] (5) -- (4);
		\draw[->] (6) -- (5);
		\draw[->] (7) -- (6);
		\draw[->] (8) -- (7);
		\draw[->] (9) -- (8);	
		\node at (2,-15) {Part of $G_B$};
		\path[->] (9) edge [loop below] (9);
\end{tikzpicture}
\caption{The nodes and arcs in $G_A$ and $G_B$ corresponding to the clause $C_j=y_1\vee y^c_3 \vee y_4$ with a total of $4$ variables.}
\label{fig:stochasticmatrices}
\end{figure}
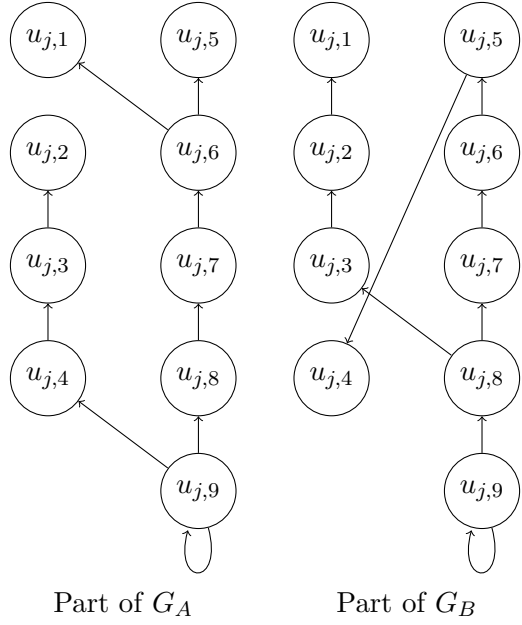

For $j \in [1:n]$, let $A^j$ ($B^j$) be the adjacency matrix of the component of $G_A$ ($G_B$) corresponding to the $j$-th clause. Since each node has in-degree 1, each column of $A^j$ ($B^j$) has exactly one entry being 1, so $A^j$ ($B^j$) is a left stochastic matrix. We can associate each truth assignment of $y_1, \ldots, y_n$ with a sequence of matrices $T^j_0, \ldots, T^j_{n-1}$ with $T^j_t \in \{A^j, B^j\}$ for $t \in [0:n-1]$. In particular, if $y_t$ is true (false), then $T^j_{t-1}$ is $A$ ($B$). Consider the product
\begin{equation*}
[0, \cdots, 0, -1] T^j_{n-1}T^j_{n-2}\cdots T^j_0
\left[
\begin{array}{c}
1\\
0\\
\vdots\\
0
\end{array}
\right].
\end{equation*}
It can be verified that this product is $-1 (0)$ if any only if the truth assignment of $y_1, \ldots, y_n$ makes clause $j$ satisfied (unsatisfied).

Order the nodes of $G_A$ or $G_B$ lexicographically, i.e.,
\begin{align*}
  u_{1,1}, u_{1,2}, \ldots, u_{1,2n+1}, u_{2,1}, \ldots, u_{2,2n+1}, \ldots, u_{m,2n+1}.
\end{align*}
Let $A$ and $B$ be the adjacency matrix of $G_A$ and $G_B$, respectively. Then both $A$ and $B$ are block diagonal matrices with $m$ blocks of $(2n+1)\times (2n+1)$ matrices. In particular,
\begin{equation} \label{eq:matricesAB}
A=\left[
\begin{array}{cccc}
A^1 & & & \\
& A^2 & & \\
& & \ddots & \\
& & & A^m
\end{array}
\right],
B=\left[
\begin{array}{cccc}
B^1 & & & \\
& B^2 & & \\
& & \ddots & \\
& & & B^m
\end{array}
\right].
\end{equation}
Both $A$ and $B$ are left stochastic matrices. When $x(0)=\sum_{j=1}^{m} e_{(j-1)(2n+1)+1}$, $c=-\sum_{j=1}^{m}e_{j(2n+1)}$, $T_t \in \{A, B\}$ for $t \in [0:n-1]$,
\begin{align*}
c^{\top}T_{n-1}\ldots T_0 x(0)=
\sum_{j=1}^m
[0, \cdots, 0, -1] T^j_{n-1}T^j_{n-2}\cdots T^j_0
\left[
\begin{array}{c}
1\\
0\\
\vdots\\
0
\end{array}
\right].
\end{align*}
Therefore, there exists a truth assignment such that the 3-SAT instance is satisfied if and only if the optimal objective value of the constructed instance of \prob{} is $-m$. This reduction is done in time polynomial in $m$ and $n$.

To prove that \prob{} is NP-hard for a pair of right stochastic matrices, we can construct an instance of \prob{} in a similar way to the case of left stochastic matrices and show that there exists a truth assignment such that the 3-SAT instance is satisfied if and only if the optimal objective value of the constructed instance is $-m$. In particular, we let $x(0)=-\sum_{j=1}^{m}e_{j(2n+1)}$ (the vector $c$ in the instance of \prob{} with left stochastic matrices above), $f(x)=c^{\top}x$ with $c=\sum_{j=1}^{m} e_{(j-1)(2n+1)+1}$ (the initial vector $x(0)$ in the instance of \prob{} with left stochastic matrices above), and the two matrices be the transpose of the two matrices $A$ and $B$ defined in~\eqref{eq:matricesAB}. 
\end{proof}

\noindent Since the matrices constructed in the proof of Theorem~\ref{thm:NPhard} are also binary matrices, we have the following result.
\begin{corollary}
\prob{} is $NP$-hard for a pair of binary matrices and a linear function $f$.
\end{corollary} 
\subsection{The Algorithm}
\label{sec:algorithm}
%In the rest of the paper, we assume that all input data are integers. 
%
%The problem with rational entries can be transformed into an equivalent problem with integer entries whose input size is polynomial in the original input size. 
%
In this section, we present a simple forward dynamic programming algorithm to solve \prob{} exactly, described in Algorithm~\ref{algorithm1}. 
The critical step of Algorithm~\ref{algorithm1} is Step~\ref{algorithm1:convexify}, which constructs $E_k$, the set of extreme points of $P_k(\Sigma, a)$, sequentially for $k = 0, 1, \ldots, K$.
\begin{algorithm}[ht]
\caption{A forward dynamic programming algorithm to solve \prob{}.}
\begin{algorithmic}[1]
\State \textbf{Input:} Matrices $\Sigma=\{A_1, \ldots, A_m\} \in \Ze^{n \times n}$, initial vector $a \in \Ze^n$, value oracle $f$, and positive integer $K$.
\State \textbf{Output:} A sequence of matrices $T_0,\ldots, T_{K-1} \in \Sigma$ that maximize $f(T_{k-1}T_{k-2}\cdots T_0 a)$.
\State \textbf{Initialize:}
Set $E_0 = \{a\}$.
\For {$k=0,1,\ldots, K-1$}
\State Set $F^i_{k} = A_iE_k$ for $i=1,\ldots, m$. \label{algorithm1:multiplication}
%\State Construct $\conv{E^A_{k+1} \cup E^B_{k+1}}$ by a convex hull algorithm. Let $E_{k+1}$ be the set of extreme points of $\conv{E^A_{k+1} \cup E^B_{k+1}}$. \label{algorithm1:convexify}
\State For each point $x \in \cup_{i=1}^mF^i_{k}$, check if $x$ is an extreme point of $\conv{\cup_{i=1}^m F^i_{k}}$, by solving a linear program. Let $E_{k+1}$ be the set of all extreme points of $\conv{\cup_{i=1}^m F^i_{k}}$.  \label{algorithm1:convexify}
%\State $E^A_{k+1} = AE_k$, $E^B_{k+1} = BE_k$, and $E_{k+1}=\emptyset$. \label{algorithm1:transform}
	%\For {$s \in E^A_{k+1} \cup E^B_{k+1}$}	
	%\State $E_{k+1} = E_{k+1} \cup \{s\}$ if $s$ is an extreme point of $\conv{E^A_{k+1} \cup E^B_{k+1}}$, by solving a linear program. \label{algorithm1:convexify}
	%\EndFor
\EndFor
\State Find an $x^*(K) \in \arg\max \{f(x) \mid x \in E_K \}$ by enumeration. \label{algorithm1:search}
\State Retrieve the optimal matrix sequence $T_{K-1},T_{K-2}, \ldots, T_0$ from $x^*(K)$. \label{algorithm1:retrieve}
\end{algorithmic}
\label{algorithm1}
\end{algorithm}

We specify the details of Step~\ref{algorithm1:convexify} later.
In fact, Step~\ref{algorithm1:convexify} can be any algorithm that takes a set of points $S$ as input and output $\ext{\conv{S}}$.
There are several efficient algorithms the construct the convex hull of a set of points on the plane, more efficient than linear programs. 
It is, however, difficult to construct $\conv{S}$ efficiently in higher dimensional space.
%
%Therefore, we only check if each point in $S$ is an extreme point of $\conv{S}$ by solving a linear program.
%
The correctness of Algorithm~\ref{algorithm1} is shown in the proposition below.
%

%\begin{definition}
%Let $\pk{k}$ be the convex hull of all possible values of $x(k)$ given the set of matrices $\Sigma$ and initial state $x(0)$ for $k\in [0:K]$, i.e.,
%\begin{align*}
%\pk{k} = \conv{\{x(k) &\mid x(k)= T_{k-1}\cdots T_0 x(0),\\
 %&T_j \in \{A,B\}, j\in [0:k-1]\}}.
%\end{align*}
%\end{definition}
%\noindent Since $x(k)$ has at most $2^k$ possible values, the set $P_k$ is a polytope.

%We illustrate the detail of the linear program in Step~\ref{algorithm1:convexify} of Algorithm~\ref{algorithm1}. To check if vector $s \in E^A_{k+1} \cup E^B_{k+1}$ is an extreme point of $\conv{E^A_{k+1} \cup E^B_{k+1}}$, we solve the following linear program in $\Re^3$.
%\begin{subequations}\label{LP_extreme point}
  %\begin{align}
  %\max_{z, z_0} \ &  s^{\top}z - z_0 \\
  %\text{s.t.} \ & p\top z - z_0 \leq 0,  \qquad p \in E^A_{k+1}\cup E^B_{k+1}, p \neq s\\
  %\ & s^{\top}z - z_0 \leq 1.
  %\end{align}
%\end{subequations}
%The point $s$ is an extreme point if and only if the optimal objective value of \eqref{LP_extreme point} is strictly positive. The optimal solution of~\eqref{LP_extreme point} gives a hyperplane (a line in $\Re^2$) to separate point $s$ and other points in $E^A_{k+1} \cup E^B_{k+1}$.

%We first prove the correctness of Algorithm~\ref{algorithm1}.
\begin{proposition}
%The set $E_k$ constructed in Algorithm~\ref{algorithm1} is the set of extreme points of $P_k(\Sigma, a)$ for each $k \in [0:K]$, and Algorithm~\ref{algorithm1} solves \prob{} correctly.
Algorithm~\ref{algorithm1} solves \prob{} correctly.
\end{proposition}
\begin{proof}
%
%We prove the first half of the statement by induction on $k$. 
%%
%When $k=0$, $E_0=P_0(\Sigma, a)=\{a\}$ and the claim holds trivially. 
%%
%Suppose that $E_t=\ext{P_t(\Sigma, a)}$ for some integer $t\ge 0$. 
%%
%Assume that $E_t$ contains $l$ points $x^1, x^2, \ldots, x^l$. 
%%
%If $q_t$ is a possible value of $x(t)$ and $q_t \notin E_t$, then $q_t$ can be expressed as a convex combination of $x^1, x^2, \ldots, x^l$. 
%%
%For each $i \in [1:m]$, the point $A_iq_t$ can be expressed as a convex combination of $A_ix^1, A_ix^2, \ldots, A_ix^l$, so $A_iq_t$ is not an extreme point of $P_{t+1}(\Sigma, a)$.
%%
%The list of potential extreme points of $P_{t+1}(\Sigma, a)$ are $A_ix^1, A_ix^2, \ldots, A_ix^l$ with $i \in [1:m]$.
%%
%Then $P_{t+1}(\Sigma, a) = \conv{\cup_{i=1}^m A_iE_t}$. 
%%
%Therefore, $E_{t+1}=\ext{P_{t+1}(\Sigma, a)}$.
%
First it is not difficult to show by induction that the set $E_k$ constructed in Algorithm~\ref{algorithm1} is the set of extreme points of $P_k(\Sigma, a)$ for each $k \in [0:K]$.
Since maximizing a convex function $f$ over a finite set $S$ is equivalent to maximizing $f$ over $\conv{S}$ as well as maximizing $f$ over $\ext{\conv{S}}$~\cite{rockafellar2015convex}, \prob{} is equivalent to $\max\{f(x) \mid x \in P_K(\Sigma, a)\} = \max\{f(x) \mid x \in E_K\}$.
Then the result follows.
\end{proof}
\begin{remark}
The fact that we are maximizing a convex function in the objective is critical for the correctness of Algorithm~\ref{algorithm1}.
If we minimize $f(x(K))$ in \prob{} instead, then Algorithm~\ref{algorithm1} will not give the correct optimal solution in general.
\end{remark}

%\subsubsection{Step~\ref{algorithm1:convexify} of Algorithm~\ref{algorithm1} for general $n$}
We now specify the linear program in Step~\ref{algorithm1:convexify} of Algorithm~\ref{algorithm1}.
%, and then analyze the time complexity of Algorithm~\ref{algorithm1}.
%
%We employ the \emph{linear programming} approach for the general $n$. 
%
Given a finite set $S=\{p^1, \ldots, p^l\} \subseteq \Re^n$, checking if a point $p^j \in S$ is an extreme point of $\conv{S}$ can be done by solving the linear program below.
\begin{subequations}
\label{eq:separation}
\begin{align} 
v^* = \max_{z, z_0} \;\; & (p^j)^{\top}z - z_0 \\
\text{s.t.} \; \; & (p^i)^{\top}z - z_0 \le 0, \; i=1,\ldots, l, i\neq j,\\
& (p^j)^{\top}z - z_0 \le 1, \\
& z \in \Re^n, z_0 \in \Re.
\end{align}  
\end{subequations}
%
%\added[comment={Define $z$ and $z_0$}]{In problem \eqref{eq:separation}, $z \in \Re^n$ and $z_0 \in \Re$ are decision variables.%that defines a hyperplane. 
%%
%A point $p^j$ is an extreme point of $\conv{S}$ if there exists a hyperplane that separates it with the set $S \setminus \{p^j\}$.
%}
%
Problem~\eqref{eq:separation} is always feasible and bounded. 
Suppose its optimal solution is $z^*$ and $ z^*_0$, and $v^*$ is the corresponding optimal objective value. If $v^* > 0$, then we find a hyperplane $(z^*)^\top x = z_0^*$ that separates $p^j$ and the set $S \setminus \{p^j\}$, so $p^j$ is an extreme point of $\conv{S}$. Otherwise $p^j$ is not an extreme point of $\conv{S}$.
%\replaced{Suppose its optimal solution is $z^*$ and $ z^*_0$, and $v^*$ is the corresponding optimal objective value. If $v^* > 0$, then we find a hyperplane $(z^*)^\top x = z_0^*$ that separates $p^j$ and the set $S \setminus \{p^j\}$, so $p^j$ is an extreme point of $\conv{S}$. Otherwise $p^j$ is not an extreme point of $\conv{S}$.}{If $v^*>0$, then $p^j$ is an extreme point of $\ext{S}$ and otherwise not.}
%
%Problem~\eqref{eq:separation} can be solved by interior point methods, say Karmarkar's algorithm, and the running time is $O(n^{3.5}L^2)$, where $L$ is the length of the binary representation of the input.
%
Problem~\eqref{eq:separation} can be solved by various interior point methods in polynomial time, for example Karmarkar's algorithm.
%
%The running time of Karmarmar's algorithm is $O(n^{3.5}L)$, where $L$ is the length of the binary representation of the input.
%
Recall that $M$ is the maximum absolute value of the entries of $A_1, \ldots, A_m$, and $a$.
%
%
%Define
%\[M= \max\{ \|A_1\|_{\infty}, \|A_2\|_{\infty}, \cdots, \|A_m\|_{\infty}, \|a\|_{\infty}\}.\]
%
\begin{proposition} \label{prop:runtime}
%\red{(Need to assume that each operation takes the same amount of time.)}
\sloppy If Karmarkar's algorithm is employed to solve the linear programs at Step~\ref{algorithm1:convexify}, the running time of Algorithm~\ref{algorithm1} is $O(m^2n^{4.5}(\log n + \log M) \sum_{k=0}^{K-1}k N_k(\Sigma)^2)$. 
\end{proposition}
\begin{proof}
We first show that the sizes of all data in Algorithm~\ref{algorithm1} are polynomial in the problem input size, which is polynomial in $K$, $n$, and $\log M$.
To see this, for any integer $k \ge 0$,
	\begin{align*}
	\|x(k)\|_\infty &= \max \{ \|A_ix(k-1)\|_\infty \mid A_i \in \Sigma\}   \le \max \{ \|A_i\|_\infty \mid A_i \in \Sigma\}  \cdot \|x(k-1)\|_\infty\\
	& \le (\max \{ \|A_i\|_\infty \mid A_i \in \Sigma \})^k  \cdot \|a\|_\infty  \le (nM)^kM.
	\end{align*}
Therefore, the size of $x(k)$ is $O(n \log\|x(k)\|_{\infty})=O(kn(\log n + \log M))$.

At Step~\ref{algorithm1:convexify} of iteration $k$, the number of operations of solving one linear program~\eqref{eq:separation} with $S= \cup_{i=1}^m F^{i}_{k}$ using Karmarkar's algorithm is $O(n^{3.5}L)$~\cite{karmarkar1984new}, where the input length 
$L = O(\sum_{i=1}^m |F^i_k| n\log\|x(k)\|_{\infty})=O(kmn(\log n+\log M)|E_k|)$.
Since we need to solve $m|E_k|$ linear programs, one for each point in $S$, the running time of Step~\ref{algorithm1:convexify} is $m|E_k| O(n^{3.5}L)= O(km^2n^{4.5}(\log n + \log M) |E_k|^2)$.
At iteration $k$, Step~\ref{algorithm1:multiplication} takes $O(mn^2)$ time, Step~\ref{algorithm1:search} takes $|E_K|$ queries to the value oracle of function $f$, 
and Step~\ref{algorithm1:retrieve} can be performed in $K$ steps if a $m$-ary tree is used to store the values of $x(k)$ for each $k$.
Therefore, the step with the dominating complexity is Step~\ref{algorithm1:convexify}, and the overall running time of Algorithm~\ref{algorithm1} is $O(m^2n^{4.5}(\log n + \log M) \sum_{k=0}^{K-1}k|E_k|^2)$.
Since $|E_k| \le N_k(\Sigma)$, the result follows.
\end{proof}

\subsubsection{Speeding up Algorithm~\ref{algorithm1} when $n=2$}
When $n=2$, there are many efficient algorithms to construct the convex hull of a set of points directly, such as Graham's scan and Jarvis's march~\cite{cormen2001introduction}.
Graham's scan constructs the convex hull of $l$ points on the plane in $O(l \log l)$ time~\cite{graham1972efficient}.
With a similar analysis as in Proposition~\ref{prop:runtime}, we have the result below.

\begin{proposition} \label{prop:runtime:n=2}
\sloppy When $n=2$ and Graham's scan is employed at Step~\ref{algorithm1:convexify} of Algorithm~\ref{algorithm1} to construct $E_{k+1}$, the running time of Algorithm~\ref{algorithm1} is $O(m \log m \sum_{k=0}^{K-1}N_k(\Sigma) + m\sum_{k=0}^{K-1}N_k(\Sigma) \log N_k(\Sigma))$. 
%When $n=2$ and the  is employed at Step~\ref{algorithm1:convexify}, Algorithm~\ref{algorithm1} terminates in $O(\sum\limits_{k=0}^{K-1} N_k \log N_k + N_K+K)$ time.
%The sizes of all data processed in Algorithm~\ref{algorithm1} are polynomial in $K$ and $\log L$.
\end{proposition}
%\begin{proof}
%%
%Similar to the proof of Proposition~\ref{prop:runtime}, we can show that the sizes of all data in Algorithm~\ref{algorithm1} are polynomial in the problem input size.
%%
%%
%Then at iteration $k$, the number of distinctive points in the set $\cup_{i=1}^m F^i_k$ is at most $m|E_k|$, so Graham's scan takes $O(m|E_k| (\log m + \log |E_k|))$ time at Step~\ref{algorithm1:convexify} to construct the convex hull of $\conv{\cup_{i=1}^m E^i_{k+1}}$. 
%%
%%Step~\ref{algorithm1:multiplication} takes $O(m)$ time. 
%%
%%Step~\ref{algorithm1:search} takes $|E_K|$ queries to the value oracle $f$. 
%%
%%Step~\ref{algorithm1:retrieve} can be performed in $K$ steps if a $m$-ary tree is used to store the values of $x(k)$ for each $k$.
%%
%Therefore, the overall running time of Algorithm~\ref{algorithm1} is 
%\begin{align*}
%\sum_{k=0}^{K-1}O(m|E_k|(\log m + \log |E_k|)) &+ \sum_{k=0}^{K-1} O(m) + |E_K| + O(K) \\
%&=O(m \log m \sum_{k=0}^{K-1}|E_k| + m\sum_{k=0}^{K-1}|E_k| \log |E_k|).
%\end{align*}
%\end{proof} 

\section{Polynomially Solvable Cases}
\label{sec:polytime}
In this section, we focus on discovering conditions on a set of matrices for which \prob{} is polynomially solvable.
Propositions~\ref{prop:runtime} and~\ref{prop:runtime:n=2} indicate that \prob{} is polynomially solvable if $N_k(\Sigma)$ is polynomial in $k$.
This motivated us to introduce the concept of the \polyv{} property in Section~\ref{sec:intro}.
%
%Recall the concept of the \polyv{} property for a set of matrices we introduced in Section~\ref{sec:intro}.
%
Recall that a set of matrices $\Sigma$ has the \polyv{} property if $N_k(\Sigma) = O(k^d)$ for some constant $d$.
The following proposition gives the detailed time complexity of our algorithms for matrices with the \polyv{} property, following directly from Proposition~\ref{prop:runtime} and~\ref{prop:runtime:n=2}.
\begin{proposition} \label{prop:polytime}
If the set of matrices $\Sigma$ in \prob{} has the \polyv{} property and $N_k(\Sigma) = O(k^d)$ for some constant $d$, then \prob{} can be solved in $O(m^2n^{4.5}K^{2d+2}(\log n + \log M))$ time for general $n$ and in $O(mK^{d+1} (\log m + \log K))$ time when $n=2$.
\end{proposition}
%
%\begin{proof}
%Since $|E_k| \le N_k(\Sigma) = O(k^d)$, the result follows immediately from Proposition~\ref{prop:runtime}. 
%\end{proof}
%
%\noindent Similarly, we have the following result for \prob{} in $\Re^2$.
%\begin{proposition} \label{prop:polytime:n=2}
%When $n=2$, if the set of matrices $\Sigma$ in \prob{} has the \polyv{} property and $N_k(\Sigma) = O(k^d)$ for some constant $d$, then \prob{} can be solved in $O(mK^{d+1} (\log m + \log K))$ time. 
%\end{proposition}
%

Thus our focus in this section is to discover conditions for a set of matrices to have the \polyv{} property.
We introduce additional notations that will be used in the rest of the paper. 
Given a set of matrices $\Sigma=\{A_1, A_2, \ldots, A_m\} \subseteq \Re^n$ and a vector $a \in \Re^n$, define
\begin{align}
X_k(\Sigma, a) &= \{x(k) \mid x(k)= T_{k-1}\cdots T_0 a, T_j \in \Sigma, j\in [0:k-1]\}\\
%P_k(\Sigma, a) &= \conv{X_k(\Sigma, a)}\\
E_k(\Sigma, a) &= \ext{P_k(\Sigma, a)} 
%N_k(\Sigma, a) &= |E_k(\Sigma, a)|\\
%N_k(\Sigma) &=\sup_{a \in \Re^n}\{N_k(\Sigma, a)\}. \label{eq:nk}
\end{align}
for each integer $k\ge 0$.
Recall that $P_k(\Sigma, a) = \conv{X_k(\Sigma, a)}$, $N_k(\Sigma, a) = |E_k(\Sigma, a)|$, and $N_k(\Sigma) =\sup_{a \in \Re^n}\{N_k(\Sigma, a)\}$.
Since $P_k(\Sigma, a)$ is the convex hull of at most $m^k$ points, both $N_k(\Sigma, a)$ and $N_k(\Sigma)$ are well defined and bounded above by $m^k$.

Some obvious cases that have the \polyv{} property include a set $\Sigma$ of $m$ pairwise commuting matrices with constant $m$ (for which $N_k(\Sigma)=O(k^{m-1})$ since there are at most $\binom{k+m-1}{m-1}$ elements in $X_k(\Sigma,a)$), and a pair of projection matrices since there are at most $2k$ elements in $X_k(\Sigma,a)$.
%
%
%\noindent In the rest of the paper, we mainly focus on the conditions for \emph{a pair of matrices} to have \polyv{} property.
%
%
%
%Similarly, a pair $\Sigma$ of project matrices has the \polyv{} property and $N_k(\Sigma)=O(k)$.
% 
%\begin{proposition} \label{prop:commproj} $\;$
%A pair $\Sigma$ of commuting matrices or project matrices has the \polyv{} property and $N_k(\Sigma)=O(k)$.
%\end{proposition}
%%
%\begin{proof}
%Let $\Sigma=\{A,B\}$. 
%%
%If $A$ and $B$ commute, then for any $a \in \Re^n$ and integer $k \ge 0$,
%$X_k(\Sigma, a) = \{A^ka, A^{k-1}Ba, \ldots, B^ka\}$.
%%
%Since $X_k(\Sigma,a)$ contains at most $k+1$ elements, $N_k(\Sigma, a)=O(k)$.
%%
%Thus $N_k(\Sigma)=O(k)$.
%%
%Now suppose that $A$ and $B$ are both projection matrices, i.e., $A^2=A$ and $B^2=B$.
%%
%Then the product of $k$ matrices with $A$ and $B$ has at most $2k$ different outcomes:
%\begin{itemize}
    %\item $A$ or $B$, if only $A$ or $B$ appears in the product sequence;
    %\item $AB$ or $BA$, if $A$ and $B$ alternate once in the product sequence;
    %\item $\cdots$;
    %\item $ABAB\cdots$ or $BABA\cdots$, if $A$ and $B$ alternate $k-1$ times in the product sequence.
  %\end{itemize}
%%
%Then $N_k(\Sigma)=O(k)$.
%\end{proof}

%\noindent We now focus on matrices in $\Re^2$.
\begin{proposition}\label{proposition_Problem_2_singular}
A set $\Sigma$ of $m$ matrices in $\Re^{n \times n}$ with at most one matrix with rank greater than one has the \polyv{} property and $N_k(\Sigma)=O(mk)$.
%\prob{} can be solved by Algorithm~\ref{algorithm1} in $O(K^2\log K)$ time.
\end{proposition}
\begin{proof}
Let $\Sigma=\{A_1, \ldots, A_m\}$.
With loss of generality, assume that no $A_i$ is the zero matrix, and $A_1, A_2, \ldots, A_{m-1}$ are of rank one.
Then for any $a \in \Re^n$ the set $A_iP_k(\Sigma, a)$ contains at most two extreme points for $i=1,\ldots, m-1$. 
For each integer $k \ge 0$, $P_{k+1}(\Sigma, a) = \conv{\cup_{i=1}^m A_iP_k(\Sigma, a)}$, so $N_{k+1}(\Sigma,a) \le \sum_{i=1}^m|\ext{A_iP_k(\Sigma, a)}| \le 2(m-1) + N_k(\Sigma,a)$.
Then $N_{k+1}(\Sigma,a) \le N_0(\Sigma, a) + 2k(m-1)$, so $N_k(\Sigma) = O(mk)$.
\end{proof}

\begin{proposition}\label{proposition_Problem_2_share_eigenvector}
A set $\Sigma$ of two $2 \times 2$ matrices that share at least one common eigenvector has the \polyv{} property and $N_k(\Sigma)=O(k)$.
%\prob{} can be solved by Algorithm~\ref{algorithm1} in $O(K^2\log K)$ time.
\end{proposition}
\begin{proof}
If matrices $A$ and $B$ in $\Sigma$ share two eigenvectors, then they commute and there are at most $k+1$ different points in $X_k(\Sigma, a)$ for any $a$.
Now suppose that $A$ and $B$ in $\Sigma$ share exactly one eigenvector $q_1$. 
Then $q_1$ must be a real vector. 
%
%Otherwise the conjugate of $q_1$ must also be the eigenvector of $A$ and $B$, and $A$ and $B$ will share two eigenvectors. 
%
Assume the corresponding eigenvalues of $q_1$ in $A$ and $B$ are $\lambda_{11}$ and $\mu_{11}$, respectively. 
Since $q_1$ is a real vector, $\lambda_{11}$ and $\mu_{11}$ are both real-valued. 
Without loss of generality, assume $\|q_1\|_2=1$. 
Let $q_2 \in \Re^2$ be a unit vector orthogonal to $q_1$. 
Consider the vector $Aq_2$. Since $q_1$ and $q_2$ form a basis of $\Re^2$, we have $Aq_2 = \lambda_{12} q_1 + \lambda_{22} q_2$ for some $\lambda_{12}, \lambda_{22} \in \Re$.
Similarly, we have $Bq_2=\mu_{12}q_1 + \mu_{22}q_2$ for some $\mu_{12}, \mu_{22}\in \Re$. 
Let $Q=\begin{bmatrix}
q_1 & q_2
\end{bmatrix},
\;
\Lambda=\begin{bmatrix}
\lambda_{11} & \lambda_{12}\\
0 & \lambda_{22}
\end{bmatrix},
\;
M=\begin{bmatrix}
\mu_{11} & \mu_{12}\\
0 & \mu_{22}
\end{bmatrix}$.
We have $\Lambda$ and $M$ as real matrices, $QQ^{\top}=I$, $A=Q\Lambda Q^{\top}$, and $B=QMQ^{\top}$.

\sloppy Any product of $k$ matrices with $A$ and $B$ can be written in the form of
$A^{l_1}B^{m_1}A^{l_2}B^{m_2}\ldots A^{l_s}B^{m_s}$
with $l_1, m_s \in \Ne$, $l_2,\ldots,l_s, m_1, \ldots, m_{s-1} >0$ for some $s\ge 1$, and $\sum_{j=1}^s(l_j + m_j) = k$.
We simplify the product as follows.
\begin{align*}
~&A^{l_1}B^{m_1}A^{l_2}B^{m_2}\ldots A^{l_s}B^{m_s} \\
%=~&(Q\Lambda Q^{\top})^{l_1}(QMQ^{\top})^{m_1}\ldots (QMQ^{\top})^{m_s}\\
%=~&Q\Lambda^{l_1}M^{m_1}\ldots \Lambda^{l_s}M^{m_s}Q^{\top}\\
=~&Q
\begin{bmatrix}
\lambda_{11}^{l_1+\ldots+l_s}\mu_{11}^{m_1+\ldots+m_s} & *\\
0 & \lambda_{22}^{l_1+\ldots+l_s}\mu_{22}^{m_1+\ldots+m_s}
\end{bmatrix}
Q^{\top}\\
=~&Q
\begin{bmatrix}
\lambda_{11}^{p}\mu_{11}^{k-p} & *\\
0 & \lambda_{22}^{p}\mu_{22}^{k-p}
\end{bmatrix}
Q^{\top},
\end{align*}
where $p=l_1+\ldots+l_s$ and $*$ represents some real number. 
Let $\Pi_p$ be the set of all matrices in the form of $\begin{bmatrix}
\lambda_{11}^{p}\mu_{11}^{k-p} & *\\
0 & \lambda_{22}^{p}\mu_{22}^{k-p}
\end{bmatrix}$
calculated from a product of $k$ matrices with $p$ matrix $A$'s and $(k-p)$ matrix $B$'s.
The set $\Pi_0$ contains one matrix in the form of
$\begin{bmatrix}
\mu_{11}^{k} & *\\
0 & \mu_{22}^{k}
\end{bmatrix}$. 
Call this matrix $C_0$.
The set $\Pi_k$ contains one matrix in the form of 
$\begin{bmatrix}
\lambda_{11}^{k} & *\\
0 & \lambda_{22}^{k}
\end{bmatrix}$.
Call this matrix $C_k$. 
For $1\le p\le k-1$, any matrix in $\Pi_p$ can be represented as a convex combination of two matrices in $\Pi_p$, the ones with the smallest and largest $*$ entries.
Call these two matrices $C_p$ and $D_p$.
Then for any $p \in [1:k-1]$, the vector 
$x(k)= A^{l_1}B^{m_1}A^{l_2}B^{m_2}\ldots A^{l_s}B^{m_s} a$
with $\sum_{j=1}^s l_j = p$ can be represented by a convex combination of $C_pa$ and $D_pa$. 
Hence $P_k(\Sigma, a) = \conv{\{C_0a, C_1a, D_1a, C_2a, D_2a, \ldots, C_ka\}}$.
%
%
%two points in $P_k(\Sigma, a)$ for $1\le p\le k-1$. 
%%
%Thus the total number of extreme points of $P_k(\Sigma, a)$ is at most $1 + 1 + 2(k-1) =2k$. 
%
Therefore $N_k(\Sigma, a) \leq 2k$ and $N_k(\Sigma)=O(k)$. 
%By Proposition~\ref{prop:complexity}, \prob{} can be solved in $O(K^2 \log K)$ time.
\end{proof}

\begin{remark} $\;$
Each right stochastic matrix has an eigenvector $(1, 1)^{\top}$. Therefore, any pair of $2 \times 2$ right stochastic matrices has the \polyv{} property and the corresponding problem \prob{} is polynomially solvable. 
%\begin{itemize}
	%\item \red{The proof above uses a special case of the Lie-Kolchin theorem.}
	%\item If $n=2$ and $A$ and $B$ share two eigenvectors, then $A$ and $B$ commute.
	%%and \prob{} can be solved by enumeration in $O(K)$ time.
%\end{itemize}
\end{remark}

Finally we present a lemma showing that the \polyv{} property is invariant under any similarity transformation.
%that will be used later in proving that a pair of $2 \times 2$ binary matrices has the \polyv{} property.
%
\begin{lemma} \label{lemma:equivalentpairs}
A finite set of $n \times n$ matrices $\Sigma$ has the \polyv{} property if and only if $S\Sigma S^{-1}$ has the \polyv{} property for any nonsingular real matrix $S$.
\end{lemma}
\begin{proof}
 It suffices to show that $N_k(\Sigma) = N_k(S\Sigma S^{-1})$.
	We claim that  $P_k(\Sigma, a) =  P_k(S\Sigma S^{-1}, Sa)$ for any $a \in \Re^n$.
	To see this, note that any extreme point $p$ of $P_k(\Sigma, a)$ can be written as $p=T_{k-1}T_{k-2}\cdots T_{0}a$ with $T_j \in \Sigma$ or $j\in [0:k-1]$.
	Then \[p=  T_{k-1}T_{k-2}\cdots T_{0}a =  S^{-1} (ST_{k-1}S^{-1})(ST_{k-2}S^{-1}) \cdots (ST_{0}S^{-1})Sa.\]
We have $p \in S^{-1} P_k(S\Sigma S^{-1}, Sa)$. 
Therefore, $P_k(\Sigma, a) \subseteq S^{-1} P_k(S\Sigma S^{-1}, Sa)$. 
\sloppy Similarly, we can show that $P_k(\Sigma, a) \supseteq S^{-1} P_k(S\Sigma S^{-1}, Sa)$, so $P_k(\Sigma, a) = S^{-1} P_k(S\Sigma S^{-1}, Sa)$. 
Since $S$ is nonsingular, the number of extreme points of $P_k(\Sigma, a)$ equals the number of extreme points of $P_k(S\Sigma S^{-1}, Sa)$, i.e., $N_k(\Sigma, a) = N_k(S\Sigma S^{-1}, Sa)$.
Thus $N_k(\Sigma) = \sup_{a \in \Re^n} N_k(\Sigma, a) 	= \sup_{a \in \Re^n} N_k(S\Sigma S^{-1}, Sa) \le	N_k(S\Sigma S^{-1})$.
By symmetry, we can show that $N_k(S\Sigma S^{-1}) \le N_k(\Sigma)$. 
Therefore, $N_k(\Sigma) = N_k(S\Sigma S^{-1})$.
\end{proof}

%%%%%%%%%%%%%%%%%%%%%
%Section: prove that the when the pair of matrices
%%%%%%%%%%%%%%%%%%%%%
\section{The $2 \times 2$ Binary Matrices}
\label{sec:binary}
%In this section, we will prove that a pair of binary matrices has the \polyv{} property.
%
%Then according to Proposition~\ref{prop:polytime:n=2}, problem \prob{} is polynomially solvable for any pair of $2 \times 2$ binary matrices.

%\begin{theorem}
 %Problem \prob{} with a pair of $2 \times 2$ binary matrices can be solved in $O(K^5 \log K)$.
%\end{theorem}
Our main result in this section is the following theorem.
\begin{theorem} \label{thm:binarymatrices}
A pair of $2 \times 2$ binary matrices has the \polyv{} property.
\end{theorem}
\noindent The seemingly innocent looking statement above is the most difficult to prove in this paper.
In fact, we are unable to provide a unified argument for all $2\times 2$ binary matrices.
This is not too surprising, however, since to the best of our knowledge there is no unified argument to show that any pair of $2 \times 2$ binary matrices has the finiteness property either~\cite{jungers2008}.
We hope that the techniques we develop in this paper can be useful in proving the \polyv{} property for other matrices in the future.

There are a total of $16$ $2 \times 2$  binary matrices, resulting in a total of $120$ different pairs of $2 \times 2$ binary matrices.
To prove Theorem~\ref{thm:binarymatrices}, we first show that the result holds for most of the 120 pairs, and then provide separate proofs for each of the remaining pairs.
Among the $16$ binary matrices, one matrix has rank zero, nine matrices have rank one, and six matrices have rank two.
The pair of matrices has the \polyv{} property if one matrix is the zero or identity matrix.
According to Proposition~\ref{proposition_Problem_2_singular}, the pair of matrices has the \polyv{} property if one matrix is singular.
Therefore, only the following five binary matrices of rank two give rise to interesting pairs:
\begin{align*}
  A_1  = \begin{bmatrix}
            0 & 1 \\
            1 & 0
          \end{bmatrix},
	~~A_2  = \begin{bmatrix}
            1 & 1 \\
            0 & 1
          \end{bmatrix},
  ~~A_3  = \begin{bmatrix}
            1 & 0 \\
            1 & 1
          \end{bmatrix},
	~~A_4  = \begin{bmatrix}
            1 & 1 \\
            1 & 0
          \end{bmatrix},
  ~~A_5  = \begin{bmatrix}
            0 & 1 \\
            1 & 1
          \end{bmatrix}.
\end{align*}

%\begin{itemize}
  %\item
  %Category 1: Matrix with two $1$'s.
%\begin{align*}
  %A_1  = \begin{bmatrix}
            %0 & 1 \\
            %1 & 0
          %\end{bmatrix};
%\end{align*}
%\item
%Category 2: Asymmetric matrices with three $1$'s.
%\begin{align*}
  %A_2  = \begin{bmatrix}
            %1 & 1 \\
            %0 & 1
          %\end{bmatrix},
  %~~A_3  = \begin{bmatrix}
            %1 & 0 \\
            %1 & 1
          %\end{bmatrix};
%\end{align*}
%\item
%Category 3: Symmetric matrices with three $1$'s.
%\begin{align*}
  %A_4  = \begin{bmatrix}
            %1 & 1 \\
            %1 & 0
          %\end{bmatrix},
  %~~A_5  = \begin{bmatrix}
            %0 & 1 \\
            %1 & 1
          %\end{bmatrix};
%\end{align*}
%\end{itemize}

\noindent The five matrices above give rise to ten different pairs of binary matrices.
Observe that
\begin{align*}
  A_1A_1A^{-1}_1 = A_1, \; A_1A_2A^{-1}_1 = A_3, \;  A_1A_4A^{-1}_1 = A_5, \; A_2A_5A^{-1}_2 = A_4.
\end{align*}
Then by Lemma \ref{lemma:equivalentpairs}, we can group the ten pairs of matrices into the following five clusters:
%\begin{center}
%\begin{tabular}{lll}
%1.  $ \{A_1, A_2\}, \{A_1, A_3\}$ & 2. $ \{A_1, A_4\}, \{A_1, A_5\}$ & \\
%3.  $ \{A_2, A_3\}$ & 4. $ \{A_4, A_5\}$ & 5. $ \{A_2, A_4\}, \{A_3, A_5\}, \{A_2, A_5\}, \{A_3, A_4\}$ 
%\end{tabular}
%\end{center}
\begin{enumerate}
  \item $ \{A_1, A_2\}, \{A_1, A_3\}$
  \item $ \{A_1, A_4\}, \{A_1, A_5\}$
  \item $ \{A_2, A_3\}$
  \item $ \{A_4, A_5\}$
  \item $ \{A_2, A_4\}, \{A_3, A_5\}, \{A_2, A_5\}, \{A_3, A_4\}$,
\end{enumerate}
and it suffices to show that one pair of matrices within each cluster has the \polyv{} property.
In the rest of this section, we are going to show separately that each of the following five pairs of matrices has the \polyv{} property.
\[\Sigma_1=\{A_1, A_2\}, \Sigma_2=\{A_1, A_4\}, \Sigma_3=\{A_2, A_3\}, \Sigma_4=\{A_4, A_5\}, \Sigma_5=\{A_2, A_4\}.\]

\noindent We first present in the table below a complete description of how $N_k(\Sigma, a)$ grows with $k$ for the five pairs of matrices, according to the location of the initial vector $a$.
\begin{table}[htb]%
\centering
\begin{tabular}{cccccc}
\hline
 & $\Sigma_1$ & $\Sigma_2$ & $\Sigma_3$ & $\Sigma_4$ & $\Sigma_5$\\
\hline
$a \in \mathcal{Q}_1 \cup \mathcal{Q}_3$ & $O(k^2)$ & $O(k)$ & $O(k)$ & $O(k)$  & $O(k)$\\
%$a\partial{\mathcal{Q}_1} \cup \partial{\mathcal{Q}_3}$ & $O(k^2)$ & $O(k)$ & $O(k)$ & $O(k)$ & $O(k)$ \\
$a \in \inte{\mathcal{Q}_2} \cup \inte{\mathcal{Q}_4}$ & $O(k^4)$ & $O(k)$ & $O(k^2)$ & $O(k^2)$ & $O(k^2)$\\
\hline
\end{tabular}
\caption{The number of extreme points $N_k(\Sigma,a)$}
\label{table:complexity}
\end{table}

\noindent The results in Table~\ref{table:complexity} show that the number of extreme points of $P_k(\Sigma, a)$ grows linearly with $k$ when the initial vector is in the first or the third quadrant for most pairs of binary matrices except $\Sigma_1$.

\begin{example}
Figure~\ref{fig:Nk_k} illustrates how the number of extreme points $N_k(\Sigma_1, a)$ changes with $k$ given different initial vector $a$'s.
For the chosen $a$'s, the growth is at most linear in $k$ for $k\le 40$.

\begin{figure}[ht]
\centering
\includegraphics[scale=0.5]{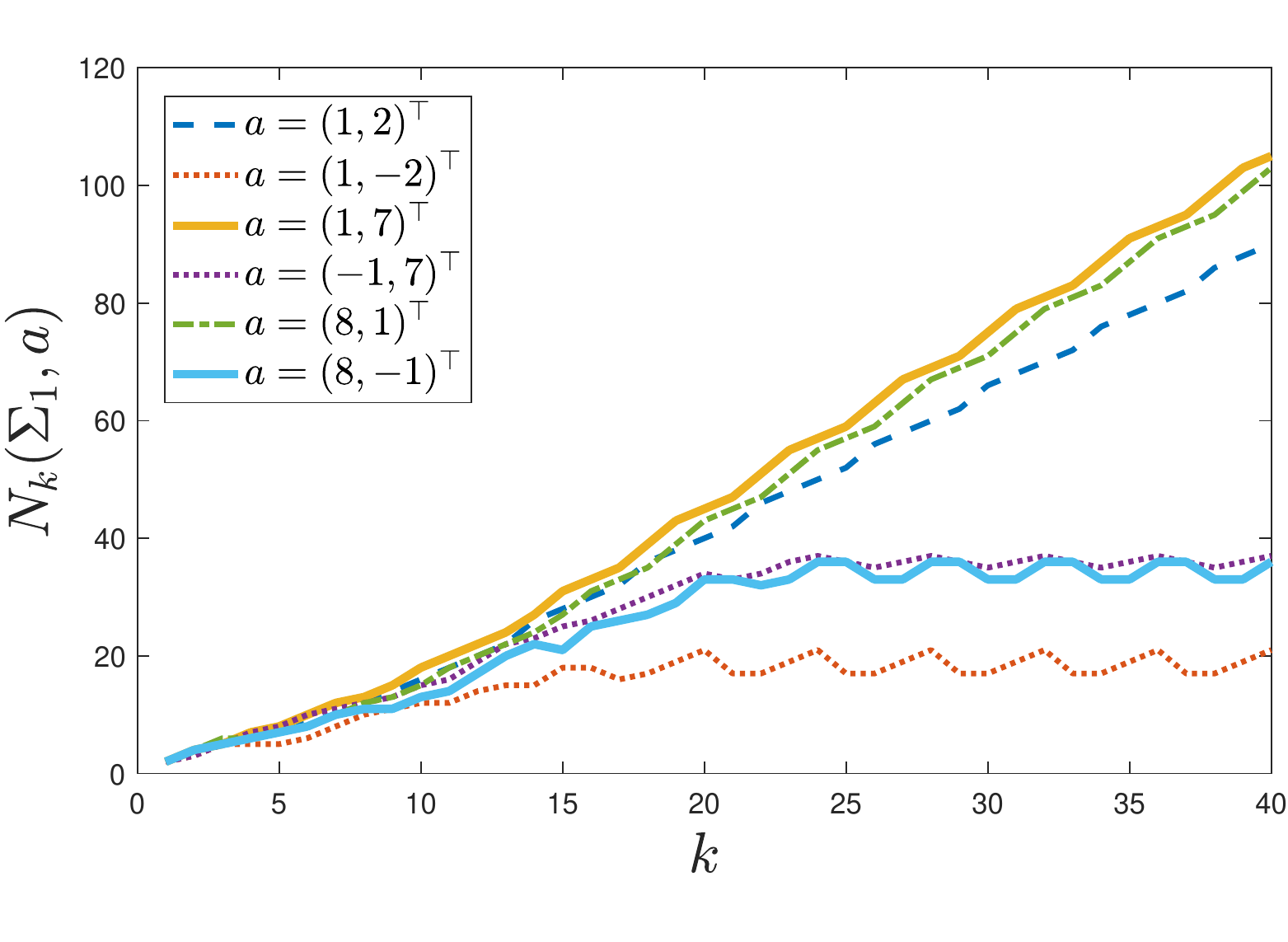}
\caption{The number of extreme points $N_k(\Sigma_1,a)$ given different initial vector $a$'s.}\label{fig:Nk_k}
\end{figure}
\end{example}

To prove the results in Table~\ref{table:complexity}, we first introduce a few notations that will be used in the rest of this section.
Given a pair $\Sigma$ of matrices and a vector $a \in \Re^2$, we divide the set of extreme points $E_k(\Sigma, a)$ of $P_k(\Sigma, a)$ into five groups.
\begin{definition} \label{def:Ei}
Let $E^i_k(\Sigma, a)$ be the set of extreme points of $P_k(\Sigma, a)$ that are maximizers of the linear program $\max\{c x \mid x \in P_k(\Sigma, a)\}$ for some $c \in \inte{\mathcal{Q}_i}$, for $i=1,2,3,4$.
Let $E^0_k(\Sigma,a)$ be the set of extreme points of $P_k(\Sigma, a)$ that are maximizers of the linear programs $\max\{c x \mid x\in P_k(\Sigma, a)\}$ where $c \in \{(1, 0), (0, 1), (-1, 0), (0,-1)\}$.
\end{definition}
\noindent Then
\begin{equation} \label{eq:extremepoints:bounds}
E_k(\Sigma, a) = \cup_{i=0}^4 E^i_k(\Sigma, a) \text{ and } N_k(\Sigma, a) \le \sum_{i=0}^4 |E^i_k(\Sigma,a)|.
\end{equation}

\begin{example}
Figure~\ref{fig:Ek} illustrates the polytopes $P_k(\Sigma_3, a)$ and the sets of extreme points $E^i_k(\Sigma_3, a)$ for $i\in [0:4]$ with $a=(2, 1)^{\top}$, for $k=5$ and $k=7$.

\begin{figure}[ht]
\centering
\subcaptionbox{$P_5(\Sigma_3,(2,1)^{\top})$ and $E^i_5(\Sigma_3,(2,1)^{\top})$}
{\includegraphics[scale=0.38]{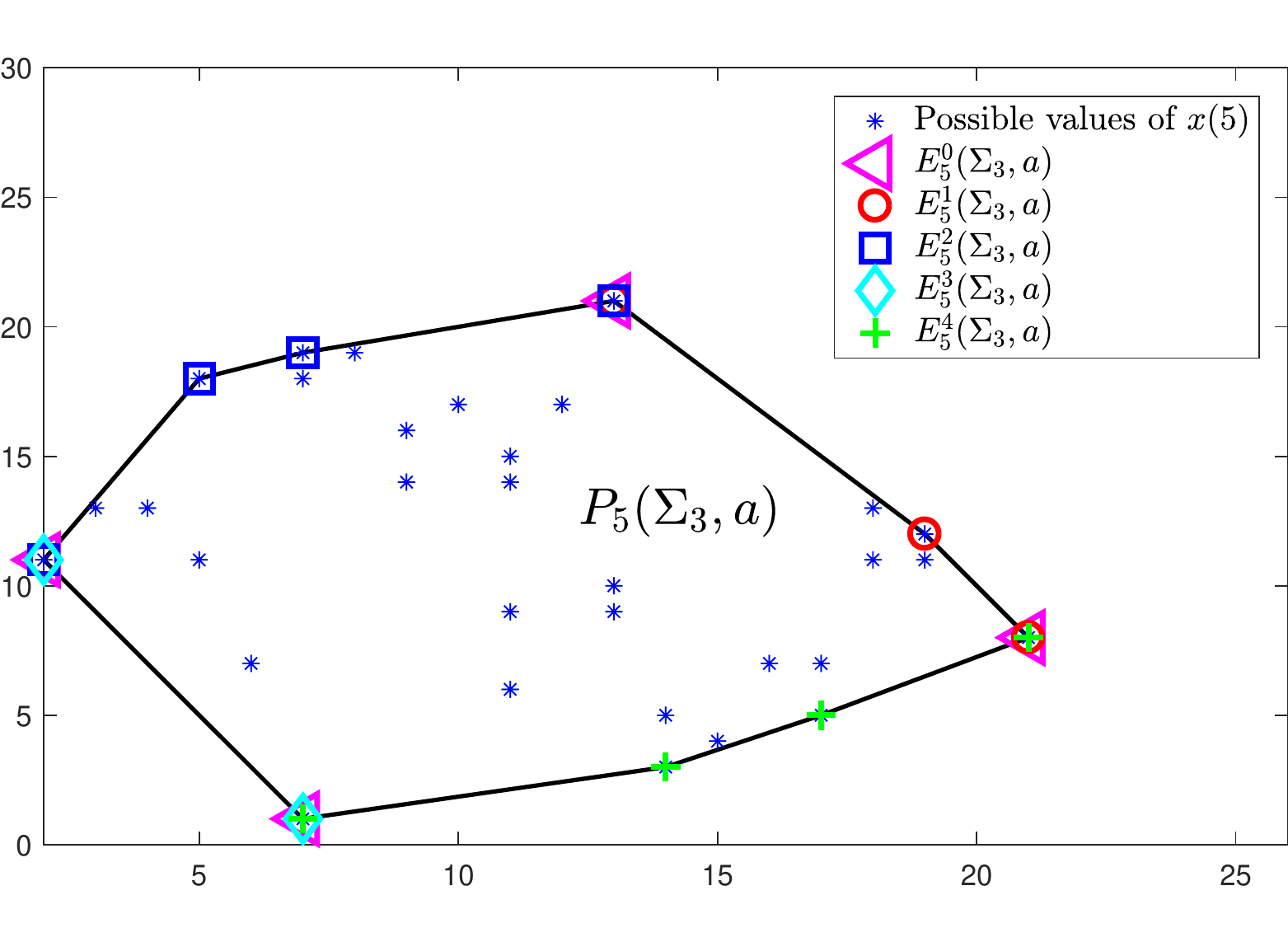}}
\hspace{2mm}
\subcaptionbox{$P_{7}(\Sigma_3,(2,1)^{\top})$ and $E^i_{7}(\Sigma_3,(2,1)^{\top})$}
{\includegraphics[scale=0.38]{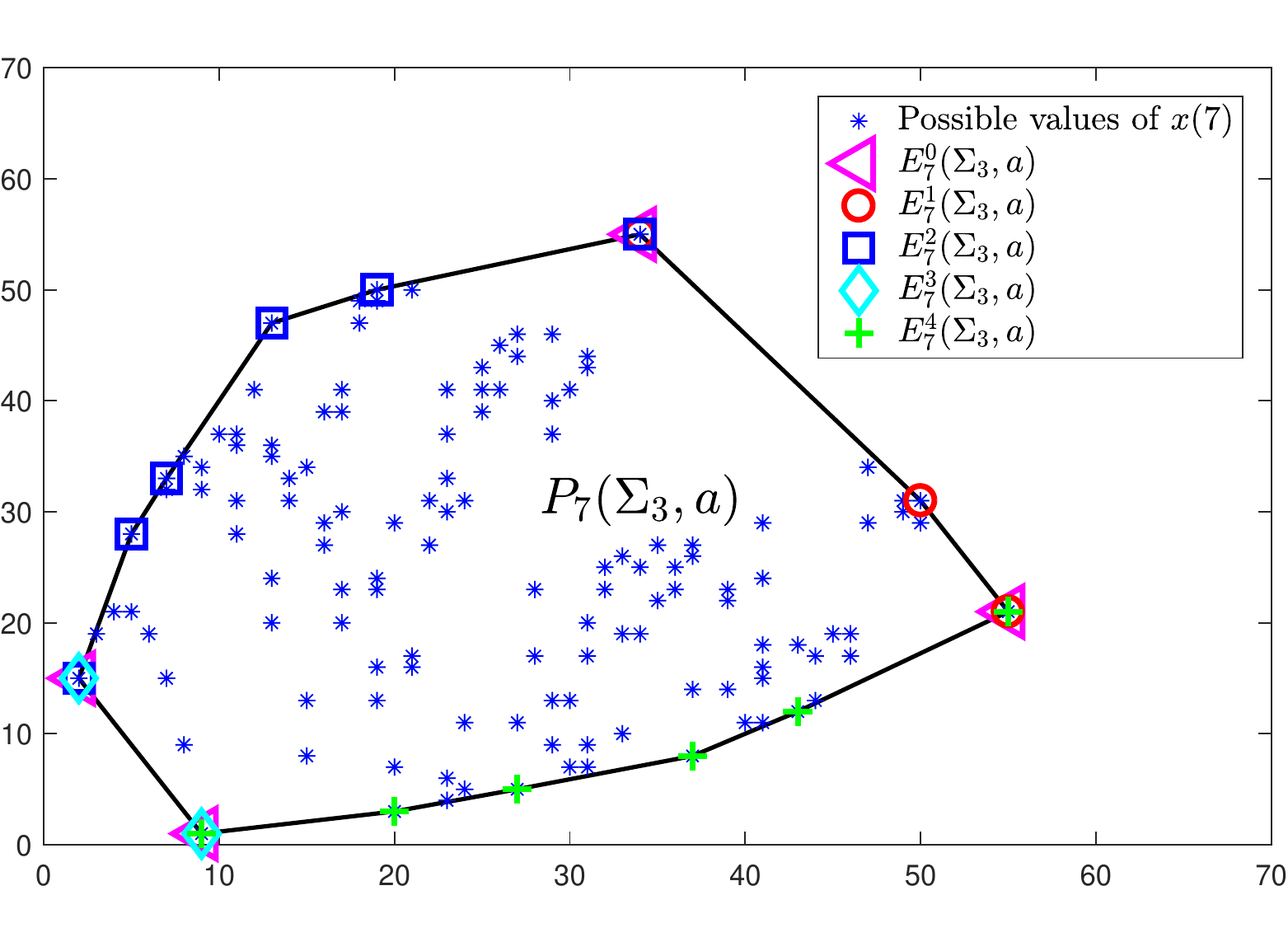}}
\caption{Examples of polytopes $P_k(\Sigma_3, a)$ and associated sets of extreme points $E^i_k(\Sigma_3,a)$ for $i \in [0:4]$.}\label{fig:Ek}
\end{figure}

%\begin{figure}[ht]
  %\centering
  %\includegraphics[width=2.7in]{Pk.png}
  %\caption{Examples of the polytope $P_k(\Sigma_3, a)$ and five sets of extreme points $E^i_k(\Sigma_3,a)$}
	%\label{fig:E_k}
%\end{figure}
\end{example}

\subsection{$\Sigma_1=\{A_1, A_2\}$}
\begin{proposition} \label{prop:A1A2}
\sloppy The pair $\Sigma_1$ has the \polyv{} property and $N_k(\Sigma_1)=O(k^4)$.
\end{proposition}

\noindent Proposition~\ref{prop:A1A2} is an immediate consequence of the following propositions.

\begin{proposition}\label{prop:A1A2:Q1}
For any $a\in \inte{\mathcal{Q}_1} \cup \inte{\mathcal{Q}_3}$, $N_k(\Sigma_1, a) = O(k^2)$.
\end{proposition}

\begin{proposition} \label{prop:A1A2:Q1boundary}
For any $a \in \partial{\mathcal{Q}_1} \cup \partial{\mathcal{Q}_3}$, $N_k(\Sigma_1, a) = O(k^2)$.
\end{proposition}

\begin{proposition}\label{prop:A1A2:Q2}
For any $a \in \inte{\mathcal{Q}_2} \cup \inte{\mathcal{Q}_4}$, $N_k(\Sigma_1, a) = O(k^4)$.
\end{proposition}

\noindent We first focus on proving Proposition~\ref{prop:A1A2:Q1}.
Our strategy is to bound the cardinality of $E^i_k(\Sigma_1, a)$ for each $i$.
Then according to~\eqref{eq:extremepoints:bounds}, $N_k(\Sigma_1)$ will be bounded as well.
\begin{lemma}\label{lemma:A1A2:Q1:Q1}
For any $a \in \inte{\mathcal{Q}_1}$ and integer $k \ge 2$, $|E^1_{k}(\Sigma_1, a)| \leq k + 1$.
\end{lemma}
\begin{proof}
To simplify the notations, we write $E^1_k$ and $P_k$ instead of $E^1_{k}(\Sigma_1, a)$ and $P_k(\Sigma_1, a)$ respectively in the rest of the proof.
We claim that $|E_{k}^1| \leq |E_{k-1}^1| + 1$ for $k\ge 2$.
Then $|E_k^1| \le |E_1^1|+ (k-1) \le 2 + (k-1) = k+1$.
To prove the claim, we first show that $E_{k}^1 \subseteq A_1E_{k-1}^1 \cup A_2E_{k-1}^1$.
Note that
\begin{align}\label{linear_program}
  \max\{c x \mid x \in P_k\} = \max\{&\max\{c A_1x \mid x \in P_{k-1}\}, \max\{c A_2x \mid x \in P_{k-1}\}\}.
\end{align}
%
%The maximizers of the linear program $\max\{c x \mid x \in P_{k}\}$ must be maximizers of the two linear programs on the right-hand side. 
%
Given $c\in \inte{\mathcal{Q}_1}$, both $c A_1$ and $c A_2$ are in the interior of $\mathcal{Q}_1$, so the maximizers of linear programs on the right are in the set $E_{k-1}^1$. Therefore, $E_{k}^1 \subseteq A_1E_{k-1}^1 \cup A_2E_{k-1}^1$.

Next we show that some points in $A_1E_{k-1}^1 \cup A_2E_{k-1}^1$ cannot belong to $E_k^1$.
Let $p = (p_1, p_2)^\top \in E_{k-1}^1$ be the maximizer of the linear program $  \max\{x_1 + x_2 \mid x \in P_{k-1}\}$ with the smallest $x_2$-coordinate.
Note that there is no other point in $E_{k-1}^1$ whose $x_2$-coordinate is $p_2$.
Otherwise suppose that there is such a point $p^\prime$.
The fact that $p$ is the maximizer of $\max\{x_1 + x_2 \mid x \in P_{k-1}\}$ implies $p_1 > p^\prime_1$.
Then $c p^\prime < c p$ for any $c \in \inte{\mathcal{Q}_1}$, which contradicts that $p^\prime \in E_{k-1}^1$.
Now we can partition $E_{k-1}^1$ into three sets $S_1=\{x \mid x \in E_{k-1}^1, x_2 > p_2\}$, $S_2=\{p\}$, and $S_3=\{x \mid x \in E_{k-1}^1, x_2 < p_2\}$.
%\[S_1=\{x \mid x \in E_{k-1}^1, x_2 > p_2\}, S_2=\{p\}, S_3=\{x \mid x \in E_{k-1}^1, x_2 < p_2\}.\]
%
Then $E_k^1 \subseteq A_1S_1 \cup A_1S_2 \cup A_1S_3 \cup A_2S_1 \cup A_2S_2 \cup A_2S_3$.
We show below that the points in $A_1S_1$ or $A_2 S_3$ cannot be in $E_{k}^1$.

First consider any point $x \in S_1$.
    \begin{itemize}
      \item If $x_1 < x_2$, we have $c A_2 x - c A_1x = c_1x_1 + c_2(x_2 - x_1) > 0$.
			%
		%The last inequality follows from that $x_1$, $c_1$, $c_2$, and $x_2-x_1$ are all strictly greater than 0.
		%
    \item Suppose $x_1 \geq x_2$ and $c_1 < c_2$. Since $x_1 + x_2 \leq p_1 + p_2$, we have $x_1 - p_1 \leq p_2 - x_2 < 0$. Therefore, $c A_1 p - c A_1x = c_1 (p_2 - x_2) + c_2(p_1 - x_1) \geq  c_1 (x_1 - p_1) + c_2( p_1 - x_1) = (c_1 - c_2) (x_1 - p_1) > 0$.
    %\begin{align*}
        %c A_1 p - c A_1x &= c_1 (p_2 - x_2) + c_2(p_1 - x_1)\\
        %&\geq  c_1 (x_1 - p_1) + c_2( p_1 - x_1)\\
        %& = (c_1 - c_2) (x_1 - p_1) > 0.
    %\end{align*}

    \item Suppose $x_1 \geq x_2$ and $c_1 \geq c_2$. Since $x_1 + x_2 \leq p_1 + p_2$, $p_2 - x_1 \ge x_2 - p_1$. Since $x_1 \geq x_2 > p_2$ and $x_1 + x_2 \leq p_1 + p_2$, we have $p_1\ge x_2$. Then $c A_2 p - c A_1x = c_1 p_1 + c_1(p_2 - x_2) + c_2(p_2 - x_1)
        \geq  c_1 p_1 + c_1 (p_2 - x_2) + c_2(x_2 - p_1) = (c_1 - c_2)(p_1 - x_2) + c_1p_2 > 0$.
        %\begin{align*}
        %c A_2 p - c A_1x &= c_1 p_1 + c_1(p_2 - x_2) + c_2(p_2 - x_1)\\
        %&\geq  c_1 p_1 + c_1 (p_2 - x_2) + c_2(x_2 - p_1)\\
        %& = (c_1 - c_2)(p_1 - x_2) + c_1p_2 > 0.
    %\end{align*}
    \end{itemize}
    Therefore, $A_1x \in A_1S_1$ cannot be a maximizer of linear program~\eqref{linear_program} with $c\in \inte{\mathcal{Q}_1}$.

Now consider any point $x \in S_3$. Since $p_2 - x_2 > 0$ and $p_1 + p_2 \ge x_1 + x_2$,
$c A_2 p - c A_2x = c_1(p_1 + p_2 - x_1 - x_2 ) + c_2(p_2 - x_2) > 0$.
%\[c A_2 p - c A_2x = c_1(p_1 + p_2 - x_1 - x_2 ) + c_2(p_2 - x_2) > 0.\]
%
Therefore, $A_2x \in A_2S_3$ cannot be a maximizer of linear program~\eqref{linear_program} with $c\in \inte{\mathcal{Q}_1}$.
Hence, $|E_{k}^1| \leq |A_1S_2| + |A_1S_3| + |A_2S_1| + |A_2S_2|  = |S_1| + |S_2| + |S_3| + |S_2| = |E_{k-1}^1| + 1$.
\end{proof}
%
%\begin{remark}
%\begin{itemize}
	%\item The strict inequalities are necessary here. If we relax it to non-strict inequality, it is not sufficient to eliminate these points as they may achieve the same objective values.
    %\item This proof is valid for any polytope which is inside the first quadrant.
%\end{itemize}
%\end{remark}

\begin{lemma}\label{lemma:A1A2:Q1:Q3}
For any $a \in \inte{\mathcal{Q}_1}$ and integer $k \ge 2$, $|E_{k}^3(\Sigma_1, a)| \leq 2$.
\end{lemma}
\begin{proof}
To simplify the notations, we write $E_k^3$ instead of $E_k^3(\Sigma_1, a)$ in the rest of the proof.
%
%Similar to the proof of~\eqref{linear_program} in Lemma~\ref{lemma:A1A2:Q1:Q1}, we can show that
%\begin{align} \label{eq:Ek--recursion}
%E_{k}^3 \subseteq A_1E_{k-1}^3 \cup A_2E_{k-1}^3
%\end{align}
%for any integer $k\ge 2$.
%
%
Let $a = (a_1, a_2)^\top \in \inte{\mathcal{Q}_1}$.
Assume that $a_1 \le a_2$.
The case in which $a_1 > a_2$ can be proved similarly.
We show below by induction that $E_k^3 \subseteq \{A_1^ka, A_1^{k-2}A_2A_1a\}$ for any $k \geq 2$.
For the base case $k=2$, given any $c \in \inte{\mathcal{Q}_3}$,
       $c A_1^2a - c A_2^2a = - 2 c_1 a_2 > 0,
       c A_1^2a - c A_1A_2a = c_1(a_1 - a_2) - c_2 a_1 > 0$.
Hence, $E_2^3 \subseteq \{A_1^2a, A_2A_1a\}$.

Now suppose that $E_t^3 \subseteq \{A_1^ta, A_1^{t-2}A_2A_1a\}$ for some $t\ge 2$.
We want to show that $E_{t+1}^3 \subseteq \{A_1^{t+1}a, A_1^{t-1}A_2A_1a\}$.
We assume that $t$ is even (a similar argument can be used to prove the result when $t$ is odd).
Similar to the proof of~\eqref{linear_program} in Lemma~\ref{lemma:A1A2:Q1:Q1}, we have
$E_{k}^3 \subseteq A_1E_{k-1}^3 \cup A_2E_{k-1}^3$ for $k\ge 2$. 
Then by the induction hypothesis, we have $E_{t+1}^3 \subseteq \{A_1^{t+1}a, A_1^{t-1}A_2A_1a, A_2A_1^{t}a, A_2A_1^{t-2}A_2A_1a \}$.
Since $t$ is even, $A_1^ta = a$ and $A_1^{t - 2}A_2A_1a = (a_1 + a_2, a_1)^{\top}$.
For any $c \in \inte{\mathcal{Q}_3}$,
      $c A_1^{t+1}a - c A_2A_1^{t}a = -c_1a_1 + c_2(a_1-a_2) > 0$, and
         $c A_1^{t+1}a - c A_2A_1^{t-2}A_2A_1a = -2c_1a_1 > 0$.

Hence, $E_{t+1}^3 \subseteq \{A_1^{t+1}a, A_1^{t-1}A_2A_1a\}$.
We conclude that $|E_k^3| \leq 2$ for any integer $k \geq 2$.
\end{proof}

 \begin{lemma}\label{lemma:A1A2:Q1:Q24}
For any $a \in \inte{\mathcal{Q}_1}$ and integer $k \ge 2$, $|E_{k}^4(\Sigma_1, a)| \leq |E_{k-1}^4(\Sigma_1, a)| + |E_{k-1}^1(\Sigma_1, a)| + 2$ and $|E_{k}^2(\Sigma_1, a)| \leq |E_{k-1}^4(\Sigma_1, a)|$.
\end{lemma}

\begin{proof}
To simplify the notations, we omit the dependence of $\Sigma_1$ and $a$ in the rest of the proof.
We first prove that $|E_{k}^4| \leq |E_{k-1}^4| + |E_{k-1}^1| + 2$.
Note that
  \begin{align*}
  \max\{c x \mid x \in P_k\} = \max\{&\max\{c A_1A_1x \mid x \in P_{k-2}\},
                                          \max\{c A_1A_2x \mid x \in P_{k-2}\},\\
                                          &\max\{c A_2A_1x \mid x \in P_{k-2}\},
                                          \max\{c A_2A_2x \mid x \in P_{k-2}\}
                                          \}.
  \end{align*}
Since $P_{k-2} \subseteq \inte{\mathcal{Q}_1}$, for any $c$ with $c_1>0$ and $c_2<0$ and $x \in P_{k-2}$,
		$c A_2^2x = (c_1, 2c_1 + c_2)x > (c_1, c_2)x = c A_1^2x,
    c A_2^2x = (c_1, 2c_1 + c_2)x > (c_2, c_1 + c_2)x = c A_1A_2x$.
Therefore,
  $\max\{c x \mid x \in P_k\} = \max\{\max\{c A_2A_1x \mid x \in P_{k-2}\},\max\{c A_2A_2x \mid x \in P_{k-2}\}  \}
 = \max\{c A_2x \mid x\in P_{k-1}\}$.
Now that $c A_2 = (c_1, c_1 + c_2)$ is a vector in the first or the fourth quadrant, the maximizers of $\max\{c x \mid x \in P_k\}$ must be in $A_2E_{k-1}^1 \cup A_2E_{k-1}^4 \cup A_2S$, where $S$ is the set of extreme points of $P_{k-1}$ that are maximizers of $\max\{x_1 \mid x \in P_{k-1}\}$.
Therefore, $|E_k^4| \leq |A_2E_{k-1}^4| + |A_2E_{k-1}^1| + |A_2S| \le |E_{k-1}^4| + |E_{k-1}^1| + 2$.

To prove that $|E_{k}^2| \leq |E_{k-1}^4|$, consider $c = (c_1, c_2)$ with $c_1 < 0$ and $c_2 > 0$.
For any $x \in P_{k-2}$,
  \begin{align*}
    c A_1A_2x &= (c_2, c_1 + c_2)x >  (c_1 + c_2, c_1)x = c A_2A_1x, \\
    c A_1A_2x &= (c_2, c_1 + c_2)x >  (c_1 , 2c_1 + c_2)x = c A_2A_2x.
  \end{align*}
Thus we have $\max\{c x \mid x \in P_k\}= \max\{\max\{c A_1A_1x \mid x \in P_{k-2}\},
                                          \max\{c A_1A_2x \mid x \in P_{k-2}\}\}
																				= \max\{c A_1x \mid x \in P_{k-1}\}$.
%\begin{align*}
%\max\{c x \mid x \in P_k\} = \max\{&\max\{c A_1A_1x \mid x \in P_{k-2}\},\\
                                          %&\max\{c A_1A_2x \mid x \in P_{k-2}\},\\
                                          %&\max\{c A_2A_1x \mid x \in P_{k-2}\},\\
                                          %&\max\{c A_2A_2x \mid x \in P_{k-2}\}
                                          %\}\\
                                  %= \max\{&\max\{c A_1A_1x \mid x \in P_{k-2}\},\\
                                          %&\max\{c A_1A_2x \mid x \in P_{k-2}\}
                                          %\}\\
                                  %= \max\{&c A_1x \mid x \in P_{k-1}\}.
%\end{align*}
%
%
Since $c A_1 = (c_2, c_1)$ is a vector in the interior of the fourth quadrant, the optimal solutions of $\max\{cx \mid x \in P_k\}$ must be in $A_1E_{k-1}^4$.
Therefore, $|E_k^2| \leq |E_{k-1}^4|$.
\end{proof}

Now we are ready to prove Proposition~\ref{prop:A1A2:Q1},
\begin{proof}[Proof of Proposition~\ref{prop:A1A2:Q1}]
We only need to prove the case where $a\in \inte{\mathcal{Q}_1}$.
When $a \in \inte{\mathcal{Q}_3}$, it is easy to verify that $N_k(\Sigma_1, a) = N_k(\Sigma_1, -a)$.
By Lemma~\ref{lemma:A1A2:Q1:Q1} and Lemma~\ref{lemma:A1A2:Q1:Q3}, we have $|E_k^1(\Sigma_1, a)| \leq k  + 1$ and $|E_k^3(\Sigma_1, a)| \leq 2$ for any $a \in \inte{\mathcal{Q}_1}$ and integer $k \ge 2$.
By Lemma \ref{lemma:A1A2:Q1:Q24}, for any $a \in \inte{\mathcal{Q}_1}$ and integer $k\ge 3$, $
  |E_k^4(\Sigma_1, a)|  \leq |E_{k-1}^4(\Sigma_1, a)| + |E_{k-1}^1(\Sigma_1, a)| + 2
              \leq |E_{k-1}^4(\Sigma_1, a)| + (k + 2) 
              \leq |E_{2}^4(\Sigma_1, a)| + \sum_{i = 2}^{k-1} (i + 3)
              \leq \frac{1}{2}k^2 + \frac{5}{2}k -3$,
%\begin{align*}
  %|E_k^4(\Sigma_1, a)| & \leq |E_{k-1}^4(\Sigma_1, a)| + |E_{k-1}^1(\Sigma_1, a)| + 2\\
             %& \leq |E_{k-1}^4(\Sigma_1, a)| + (k + 2)\\
             %& \leq |E_{k-2}^4(\Sigma_1, a)| + (k + 1) + (k + 2)\\
             %& \qquad \qquad \cdots\\
             %& \leq |E_{2}^4(\Sigma_1, a)| + \sum_{i = 2}^{k-1} (i + 3)\\
             %& \leq 4 + \frac{1}{2}(k + 7)(k - 2)\\
			 %& = \frac{1}{2}k^2 + \frac{5}{2}k -3,
%\end{align*}
and $|E_k^2(\Sigma_1, a)| \le |E_{k-1}^4(\Sigma_1, a)| \le \frac{1}{2}k^2 + \frac{3}{2}k -5$.
Therefore, $N_k(\Sigma_1, a)  \leq |E_k^1(\Sigma_1, a)| + |E_k^2(\Sigma_1, a)| + |E_k^3(\Sigma_1, a)| +|E_k^4(\Sigma_1, a)| + |E_k^0(\Sigma_1, a)| = O(k^2)$.
\end{proof}
\noindent The conclusion $N_k(\Sigma_1, a)=O(k^2)$ can be easily extended to the case where $a$ is on the boundary of the first or third quadrant.
\begin{proof}[Proof of Proposition~\ref{prop:A1A2:Q1boundary}]
We only need to prove the case where $a \in \partial{\mathcal{Q}_1}$.
The case where $a \in \partial{\mathcal{Q}_3}$ follows from the fact $N_k(\Sigma_1, a)=N_k(\Sigma_1, -a)$.
We first prove the result when $a$ is on the positive $x_1$-axis.
Without loss of generality, assume that $a = (1,0)^\top$.
We claim that for any integer $k \ge 3$,
\[X_k(\Sigma_1, (1,0)^{\top}) = X_{k-2}(\Sigma_1, (1,1)^{\top}) \cup \{(1,0)^\top, (0,1)^\top\}.\]
To see this, consider any value of $x(k)$ in $X_k(\Sigma_1, (1,0)^{\top})$ that is different from $(1,0)^\top$ and $(0,1)^\top$.
Since $A_1^ta = (0,1)^\top$ for odd integer $t\ge 1$, $A_1^ta =(1,0)^\top$ for even integer $t \ge 1$, $A_2^ta  = (1,0)^\top$ for any integer $t \ge 1$, and $A_2A_1a=(1,1)^{\top}$.
For $x(k)$ to take a value different from $(0,1)^\top$ and $(1,0)^\top$, $x(k)$ must be in the form of $T_{k-1}\cdots T_l x(l)$ with $T_j \in \Sigma_1$ for $j\in [l:k-1]$ and $x(l)=(1,1)^{\top}$ for some $l \ge 2$.
But when $x(l)=(1,1)^{\top}$, we have $A_1^jx(l) = x(l)$ for any integer $j \ge 1$. Then $x(k) = T_{k-1}\cdots T_l A_1^{l-2}x(l)$, which is a point in $X_{k-2}(\Sigma_1, (1,1)^{\top})$.
Thus $X_k(\Sigma_1, (1,0)^{\top}) \subseteq X_{k-2}(\Sigma_1, (1,1)^{\top}) \cup \{(1,0)^\top, (0,1)^\top\}$.
On the other hand, given a point in $X_{k-2}(\Sigma_1, (1,1)^{\top})$ written in the form of $T_{k-3}\cdots T_0 (1,1)^{\top}$ with $T_j \in \Sigma_1$ for $j \in [0:k-3]$, we can also write it in the form of $T_{k-3}\cdots T_0 A_2 A_1 (1,0)^{\top}$.
Thus $X_k(\Sigma_1, (1,0)^{\top}) \supseteq X_{k-2}(\Sigma_1, (1,1)^{\top}) \cup \{(1,0)^\top, (0,1)^\top\}$.
Therefore, $N_k(\Sigma_1, (1,0)^{\top}) \le N_{k-2}(\Sigma_1, (1,1)^{\top}) + 2 = O(k^2)$.
The last equality follows from Proposition~\ref{prop:A1A2:Q1}.
The case where $a$ is on the positive $x_2$-axis can be proved similarly.
%
%When $x(0)$ is on the negative $x_2$-axis, observe that $X_k(\{A_1, A_2\}, x(0)) = - X_k(\{A_1, A_2\}, -x(0))$ and the result also follows.
\end{proof}

We proceed to prove Proposition~\ref{prop:A1A2:Q2}.
Let $X^{2,4}_k(\Sigma_1, a)$ be the set of points in $X_k(\Sigma_1,a)$ that are in the interior of the second or fourth quadrant, i.e.,
\[X^{2,4}_k(\Sigma_1, a) = X_k(\Sigma_1, a) \cap (\inte{\mathcal{Q}_2} \cup \inte{\mathcal{Q}_4)}.\]
\begin{lemma}\label{lemma:A1A2:Q4}
For any $a \in \inte{\mathcal{Q}_4}$ and integer $k \geq 2$, $X^{2,4}_k(\Sigma_1, a)$ contains no more than $4k+4$ points.
\end{lemma}
\begin{proof}
Without loss of generality, assume $a = (1, a_2)^{\top}$ with $a_2 < 0$.
Let $u_0 = \max\{1, -a_2\}$ and $v_0 = \min\{1, -a_2\}$.
Define the following sequence of non-negative numbers recursively:
$u_{j} = \max\{v_{j-1}, u_{j-1} - v_{j-1}\}$ and $v_{j} = \min\{v_{j-1}, u_{j-1} - v_{j-1}\}$ for $j \in [1:k]$.
\sloppy For each $t \in [0:k]$, define $S_t=\{(u_t, -v_t)^\top, (-u_t, v_t)^\top, (v_t, -u_t)^\top, (-v_t, u_t)^\top\}$.

Given any $s^k \in X^{2,4}_k(\Sigma_1, a)$, assume that $s^k=T_{k-1}\cdots T_0a$ with $T_j \in \Sigma_1$ for $j \in [0:k-1]$.
We claim that for any integer $k \ge 0$, if $t$ out of the $k$ matrices $T_0, \cdots, T_{k-1}$ are $A_2$, then $s^k \in S_t$.
We prove the claim by induction on $k$.
First consider the base case $k=0$.
If $|a_2|\ge 1$, then $u_0=-a_2$ and $v_0=1$, so $s^k=a=(v_0, -u_0)^{\top} \in S_0$.
If $|a_2|<1$, then $u_0=1$ and $v_0=-a_2$, so $s^k=a=(u_0, -v_0)^{\top} \in S_0$.
Now suppose that the claim holds for integer $k=l \ge 0$. Specifically, $s^l=T_{l-1}\cdots T_0 a \in S_t$ if $t \in [0:l]$ out of the $l$ matrices $T_0, \cdots, T_{l-1}$ are $A_2$.
We want to prove that any point $s^{l+1}= T_{l+1}\cdots T_0 a$ in $X^{2,4}_{l+1}(\Sigma_1, a)$ also belongs to $S_t$, if $t \in [0:l+1]$ out of the $l+1$ matrices $T_0, \cdots, T_{l+1}$ are $A_2$.
If $T_{l+1} = A_1$, then $t$ out of the $l$ matrices $T_l, \ldots, T_0$ are $A_2$.
Based on the induction hypothesis, the point $s=T_l\cdots T_0a \in S_t$.
Since $A_1S_t = S_t$, $s^{l+1}=A_1s$ must be in $S_t$ as well.
If $T_{l+1} = A_2$, then $(t-1)$ out of the $l$ matrices $T_l, \ldots, T_0$ are $A_2$.
Based on the induction hypothesis, the point $s=T_l\cdots T_0a \in S_{t-1}$.
The set $S_t$ contains four points.
We consider one case $s=(u_{t-1}, -v_{t-1})^\top$ here, and the result for the other cases can be proved similarly.
We have $s^{l+1}=A_2s=(u_{t-1}-v_{t-1}, -v_{t-1})^{\top}$.
Since $s^{l+1}$ is in the interior of second or fourth quadrant and $-v_{t-1}<0$, we must have $u_{t-1}-v_{t-1} >0$.
If $v_{t-1} \ge u_{t-1}-v_{t-1}$, then $u_t= v_{t-1}$, $v_t=u_{t-1}-v_{t-1}$, and $s^{l+1}=(v_t, -u_t)^{\top} \in S_t$.
If $v_{t-1} < u_{t-1}-v_{t-1}$, then $u_t= u_{t-1}-v_{t-1}$, $v_t=v_{t-1}$, and $s^{l+1}=(u_t, -v_t)^{\top} \in S_t$.
With the claim, we conclude that $X^{2,4}_k(\Sigma_1, a)$ contains at most $4k+4$ different points.
\end{proof}

\begin{proof}[Proof of Proposition~\ref{prop:A1A2:Q2}]
We omit the dependence of $\Sigma_1$ in the rest of the proof to simplify the notation.
Given a set $S \subseteq \Re^2$, define $X_k(S) = \cup_{a \in S} X_k(a)$.
First note that for any $x$ in the first (third) quadrant, $A_1x$ and $A_2x$ are both in the first (third) quadrant.
Thus the points in $X^{2,4}_{i+1}(a)$ can only be linear transformations of points in $X^{2,4}_{i}(a)$ under $A_1$ or $A_2$.
In addition, for any $x$ in the second or fourth quadrant, $A_1x$ is also in the second or fourth quadrant.
Therefore, for any integer $i \ge 0$,
$A_1X^{2,4}_i(a) \cup A_2X^{2,4}_i(a) = X^{2,4}_{i+1}(a) \cup (A_2X^{2,4}_i(a) \cap (\mathcal{Q}_1 \cup \mathcal{Q}_3))$.
Given any $a$ in the interior of the second quadrant, we have
%\begin{equation}
%\begin{split}
%X_k(a) &= X_k(X^{2,4}_0(a))=X_{k-1}(A_1X^{2,4}_0(a) \cup A_2X^{2,4}_0(a)) \\
%&=X_{k-1}(X^{2,4}_1(a)) \cup X_{k-1}(A_2X^{2,4}_0(a) \cap (\mathcal{Q}_1 \cup \mathcal{Q}_3)) \\
%&=(X_{k-2}(X^{2,4}_2(a)) \cup X_{k-2}(A_2X^{2,4}_1(a) \cap (\mathcal{Q}_1 \cup \mathcal{Q}_3))) \cup X_{k-1}(A_2X^{2,4}_0(a) \cap (\mathcal{Q}_1 \cup \mathcal{Q}_3)) \\
%& \qquad \cdots \\
%&= X_{l}(X^{2,4}_{k-l}(a)) \cup \cup_{j=l}^{k-1} X_{j}(A_2 X^{2,4}_{k-1-j}(a) \cap (\mathcal{Q}_1 \cup \mathcal{Q}_3)) 
%\end{split}
%\label{eq:A1A2:Q2}
%\end{equation}
\begin{equation}
\begin{split}
X_k(a) = &X_k(X^{2,4}_0(a))=X_{k-1}(A_1X^{2,4}_0(a) \cup A_2X^{2,4}_0(a)) \\
= &X_{k-1}(X^{2,4}_1(a)) \cup X_{k-1}(A_2X^{2,4}_0(a) \cap (\mathcal{Q}_1 \cup \mathcal{Q}_3)) \\
= &(X_{k-2}(X^{2,4}_2(a)) \cup X_{k-2}(A_2X^{2,4}_1(a) \cap (\mathcal{Q}_1 \cup \mathcal{Q}_3)))   \\ 
&\cup X_{k-1}(A_2X^{2,4}_0(a) \cap (\mathcal{Q}_1 \cup \mathcal{Q}_3)) \\
& \qquad \cdots \\
= &X_{l}(X^{2,4}_{k-l}(a)) \cup \cup_{j=l}^{k-1} X_{j}(A_2 X^{2,4}_{k-1-j}(a) \cap (\mathcal{Q}_1 \cup \mathcal{Q}_3)) 
\end{split}
\label{eq:A1A2:Q2}
\end{equation}
for any $l \ge 1$.

On the other hand, for any $x$ in the first or third quadrant, we have shown that there exists some integer $k_0$ and $\alpha>0$ such that $N_k(x) \le \alpha k^2$ for any integer $k \ge k_0$.
Setting $l=k_0$ in equation~\eqref{eq:A1A2:Q2} we have $X_k(a)= X_{k_0}(X^{2,4}_{k-k_0}(a)) \cup \cup_{j=k_0}^{k-1} X_{j}(A_2 X^{2,4}_{k-1-j}(a) \cap (\mathcal{Q}_1 \cup \mathcal{Q}_3))$.
Therefore,
\begin{align*}
N_k(a) &\le  |X_{k_0}(X^{2,4}_{k-k_0}(a))| + \sum_{j=k_0}^{k-1} \sum_{x \in A_2 X^{2,4}_{k-1-j}(a) \cap (\mathcal{Q}_1 \cup \mathcal{Q}_3)} N_j(x)\\
& \le \sum_{x\in X^{2,4}_{k-k_0}(a)} |X_{k_0}(x)| + \sum_{j=k_0}^{k-1} |A_2 X^{2,4}_{k-1-j}(a)| \alpha j^2\\
& \le |X^{2,4}_{k-k_0}(a)| 2^{k_0} + \sum_{j=k_0}^{k-1} | X^{2,4}_{k-1-j}(a)|\cdot \alpha j^2\\
& \le (4k-4k_0+4)2^{k_0} + \sum_{j=k_0}^{k-1} (4k-4j) \alpha j^2 \le \beta k^4,
\end{align*}
for some constant $\beta$.
The second last inequality follows from Lemma~\ref{lemma:A1A2:Q4}.
Therefore $N_k(a)=O(k^4)$.
\end{proof}

\subsection{$\Sigma_2=\{A_1, A_4\}$} 
In this section, we will prove that $N_k(\Sigma_2) = O(k)$.
\begin{lemma}\label{lemma:A1A4}
Given any polytope $P \subseteq  \Re \times \Re_+$ or $P \subseteq  \Re \times \Re_-$, the number of extreme points of $\conv{P \cup A_2P}$ is at most two more than the number of extreme points of $P$.
%, i.e., $|\ext{\conv{P \cup A_2P}}| \le |\ext{P}|+2$.
\end{lemma}

\begin{proof}
We first prove the case in which $P \subseteq  \Re \times \Re_+$.
The result is easy to show if $P$ is a singleton or a line segment.
Now suppose $P$ is full dimensional.
Let $r=(r_1, r_2)^{\top}$ be the extreme point of $P$ with the largest $x_2$-coordinate; if there are two such extreme points, let $r$ be the one with a larger $x_1$-coordinate.
Similarly, let $s=(s_1, s_2)^{\top}$ be the extreme point of $P$ with the smallest $x_2$-coordinate; let $s$ be the one with a larger $x_1$-coordinate if there are two such extreme points.
Divide the extreme points of $P$ other than $r$ and $s$ into two sets: (1) Set $Q_1$ consisting of extreme points visited if we walk clockwise along the boundary of $P$ from $s$ to $r$; (2) Set $Q_2$ consisting of extreme points visited if we walk clockwise along the boundary of $P$ from $r$ to $s$.
Let $R = \{r,s, A_2r, A_2s\}$. Since $\ext{P}=Q_1\cup Q_2 \cup \{r,s\}$, the possible extreme points of $\conv{P \cup A_2P}$ are among $Q_1$, $Q_2$, $A_2Q_1$, $A_2Q_2$, and $R$.

We claim that any point in $Q_2$ can be represented as a convex combination of points in $Q_1 \cup A_2Q_2 \cup R$.
To see this, first consider any point $p = (p_1, p_2)^\top \in Q_2$.
By the definition of $Q_2$, we have $p_2>0$ and there exists a point $h = (h_1, h_2)^\top$ on the line segment connecting $r$ and $s$ such that $h_1 < p_1$ and $h_2 = p_2$.
See the illustration in Figure~\ref{fig:A1A4}.
We can verify that $p=\lambda A_2p+ (1-\lambda)h$ with $\lambda = \frac{p_1 - h_1}{p_1 + p_2 - h_1} \in (0,1)$.
Thus $p$ can be represented as a convex combination of $A_2p$ and $h$.
Since $h$ can also be represented by a convex combination of $r$ and $s$, $p$ can be represented as a convex combination of $A_2p$, $r$, and $s$.
Therefore, we show that any point in $Q_2$ is a convex combination of points in $Q_1 \cup A_2Q_2 \cup R$.
\begin{figure}[ht]
  \centering
  \includegraphics[width=2.7in]{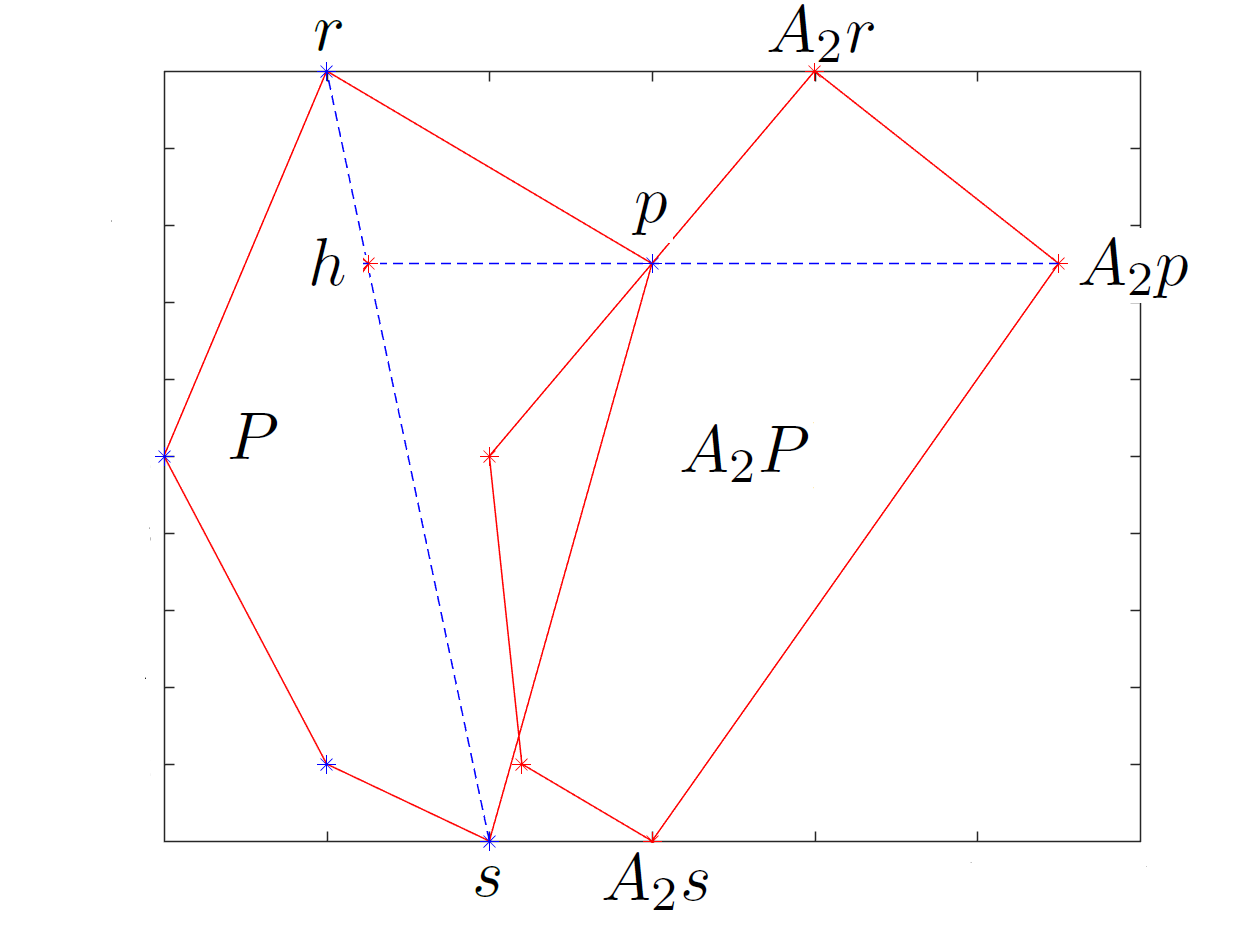}
  \caption{Point $p$ is a convex combination of $r$, $s$, and $A_2p$.}
	\label{fig:A1A4}
\end{figure}

%Next we claim that any point in $A_2Q_1$ can be represented as a convex combination of points in $Q_1 \cup A_2Q_2 \cup R$.
%%
%Consider any point $q = (q_1, q_2)^\top \in Q_1$.
%%
%If $q_2=0$, then $A_2q=q$ and the claim holds.
%%
%Otherwise $q_2>0$.
%%
%By the definition of $Q_1$, there is a point $i = (i_1, i_2)^\top$ connecting $r$ and $s$ such that $i_1 > q_1$ and $i_2 = q_2$.
%%
%We can verify that $A_2q = \lambda A_2i + (1-\lambda)q$ with $\lambda = \frac{q_2}{q_2+ i_1 - q_1} \in (0,1)$.
%%
%Thus $A_2q$ is a convex combination of $A_2i$ and $q$.
%%
%Since $A_2i$ can be represented by a convex combination of $A_2r$ and $A_2s$, $q$ can be represented as a convex combination of $q$, $A_2r$, and $A_2s$.
%%
%Therefore, we show that any point in $A_2Q_1$ is a convex combination of points in $Q_1 \cup A_2Q_2 \cup R$.
%

Similarly, we can show that any point in $A_2Q_1$ is a convex combination of points in $Q_1 \cup A_2Q_2 \cup R$.
Then we have $\conv{P \cup A_2P} = \conv{Q_1 \cup A_2Q_2 \cup R}$. Thus
$|\ext{\conv{P \cup A_2P}}| = |\ext{\conv{Q_1 \cup A_2Q_2 \cup R}}| 
                               \leq |Q_1| + |A_2Q_2| + |R|
															 \le (|Q_1| + |Q_2| + |\{r, s\}|) + 2 
                               \leq |\ext{P}| + 2$.
The result for any $P \subseteq \Re \times \Re_-$ can be proved similarly.
\end{proof}

\begin{proposition}
The pair $\Sigma_2$ has the \polyv{} property and $N_k(\Sigma_2) = O(k)$.											
\end{proposition}

\begin{proof}
To simplify the notation, we omit the dependence of $\Sigma_2$ in $N_k(\Sigma_2, a)$ and $P_k(\Sigma_2, a)$ in the rest of this proof.
We claim that $N_{k+1}(a) \leq N_k(a) + 8$ for any $a \in \Re^2$ and integer $k \ge 2$.
Then $N_k(a) \le N_{k-1}(a)+8 \le \cdots \le N_2(a) + 8(k-2) \le 4+8(k-2) = 8k-12$.
Thus $N_k = O(k)$.

To prove the claim, first observe that $P_{k+1}(a) = \conv{A_1P_k(a) \cup A_4P_k(a)} = \conv{A_1P_k(a) \cup A_2A_1P_k(a)}$.
%Since $A_1$ is nonsingular, the number of extreme points of $A_1P_k$ is $N_k$.
%
Define $P^+ = A_1P_k(a) \cap \{x \mid x_2 \geq 0\}$ and $P^- = A_1P_k(a) \cap \{x \mid x_2 \leq 0\}$.
Notice that $P^+$ is a polytope in $\Re\times \Re_+$ and $P^-$ is polytope in $\Re \times \Re_-$, and $|\ext{P^+}| +  |\ext{P^-}| \leq |\ext{A_1P_k(a)}| + 4 = N_k(a) + 4$.
The first inequality above follows from the fact the line $x_2=0$ may introduce two new extreme points for both $P^+$ and $P^-$.
On the other hand,
\begin{align*}
 P_{k+1}(a) &=\conv{A_1P_k(a) \cup A_2A_1P_k(a)}\\
&=\conv{P^+ \cup P^- \cup A_2(P^+ \cup P^-)}\\
&=\conv{\conv{P^+ \cup A_2P^+} \cup \conv{P^- \cup A_2P^-}}.
\end{align*}
Thus $N_{k+1}(a) \le |\ext{\conv{\conv{P^+ \cup A_2P^+}}}| + |\ext{\conv{P^- \cup A_2P^-}}|$.
%the number of extreme points of $P_{k+1}$ is upper bounded by the sum of the number of extreme points of $\conv{P^+ \cup A_2P^+}$ and $\conv{P^- \cup A_2P^-}$.
%
By Lemma \ref{lemma:A1A4}, $|\ext{\conv{P^+ \cup A_2P^+}}| \leq |\ext{P^+}| + 2$ and $|\ext{\conv{P^- \cup A_2P^-}}| \leq |\ext{P^-}| + 2$.
Thus we have $N_{k+1}(a) \leq |\ext{\conv{P^+}}| + 2 + |\ext{\conv{P^-}}| + 2 \leq N_k(a) + 8$.
\end{proof}

\subsection{$\Sigma_3=\{A_2, A_3\}$}
We first prove the following result when the initial vector $a$ is in the first quadrant.
\begin{proposition}\label{prop:A2A3:Q1}
For any $a \in \mathcal{Q}_1$, $N_k(\Sigma_3, a) = O(k)$.
\end{proposition}
\begin{proof}
The proof is similar to the proof of Proposition~\ref{prop:A1A2:Q1} for $\Sigma_1$.
We first bound the cardinality of $E_k^i(\Sigma_3, a)$ for each $i$.
\sloppy Similar to the proofs of Lemmas~\ref{lemma:A1A2:Q1:Q1},~\ref{lemma:A1A2:Q1:Q3}, and~\ref{lemma:A1A2:Q1:Q24}, we can show that for any $a \in \mathcal{Q}_1$, $|E^1_k(\Sigma_3,a)| \le 4$ when $k \ge 3$, $E^3_k(\Sigma_3,a)  \subseteq \{A_2^k a, A_3^k a\}$ when $k\ge 1$, 
$|E^4_k(\Sigma_3, a)|  \leq |E^4_{k-1}(\Sigma_3, a)| + |E^1_{k-1}(\Sigma_3, a)| + 2$ and 
$|E^2_k(\Sigma_3, a)|  \leq |E^2_{k-1}(\Sigma_3, a)| + |E^1_{k-1}(\Sigma_3, a)| + 2$ when $k \ge 1$, respectively.
Then $|E^4_k(\Sigma_3,a)|  \leq  |E^4_{k-1}(\Sigma_3,a)| + 6 \leq |E^4_{3}(\Sigma_3,a)| + 6(k-3) \leq 6k-10$. 
Similarly, $|E^2_k(\Sigma_3,a)| \leq 6k -10 $.
Finally, for any $a \in \mathcal{Q}_1$ and integer $k \geq 3$, $N_k(\Sigma_3,a)  \leq \sum_{i=0}^4|E^i_k(\Sigma_3,a)| \leq 8 + 4 + (6k-10) + 2 + (6k - 10) + 8= 12k - 6$.
\end{proof}

\begin{proposition} \label{prop:A2A3}
The pair $\Sigma_3$ has the \polyv{} property and $N_k(\Sigma_3) = O(k^2)$.													
\end{proposition}

\begin{proof}
%Proposition~\ref{prop:A2A3:Q1} has shown that $N_k(\Sigma_3, a)=O(k)$ for any $a \in \mathcal{Q}_1$.
%
%Since $N_k(\Sigma_3, a)= N_k(\Sigma_3, -a)$, 
We only need to prove that $N_k(\Sigma_3, a) = O(k^2)$ for any $a \in \inte{\mathcal{Q}_4}$.
Define $f_k =\sup\{N_k(\Sigma_3, a) \mid a \in \inte{\mathcal{Q}_4}\}$ for any integer $k \ge 1$.
Note that $f_k=\sup\{N_k(\Sigma_3, a) \mid a \in \inte{\mathcal{Q}_2}\}$ for $k\ge 1$ as well.
%
%Define
%\begin{align*}
%P^{A_2}_k(\Sigma_3, a) &=\mathrm{conv}(\{x(k) \mid x(k)=T_{k-1}\cdots T_1 A_2a, T_j \in \Sigma_3, j\in [1:k-1]\}),\\
%P^{A_3}_k(\Sigma_3, a) &=\mathrm{conv}(\{x(k) \mid x(k)=T_{k-1}\cdots T_1 A_3a, T_j \in \Sigma_3, j\in [1:k-1]\}),
%\end{align*}
%for any integer $k\ge 2$.
%%
%Then $P_k(\Sigma_3, a)=\conv{P^{A_2}_k(\Sigma_3, a) \cup P^{A_3}_k(\Sigma_3, a)}$.
%%
%Thus the number of extreme points of $P_k(\Sigma_3, a)$ is bounded above by the sum of the number of extreme points of $P^{A_2}_k(\Sigma_3, a)$ and the number of extreme points of $P^{A_3}_k(\Sigma_3, a)$.
%%
%On the other hand, $P^{A_2}_k(\Sigma_3, a) = P_{k-1}(\Sigma_3, A_2a)$ and $P^{A_3}_k(\Sigma_3, a) = P_{k-1}(\Sigma_3, A_3a)$.
%
%
\sloppy Since $P_k(\Sigma_3, a)=\conv{P_{k-1}(\Sigma_3, A_2a) \cup P_{k-1}(\Sigma_3, A_3a)}$, we have $N_k(\Sigma_3, a) \le N_{k-1}(\Sigma_3, A_2a) + N_{k-1}(\Sigma_3, A_3a)$.
Consider a vector $a = (a_1, a_2)^{\top} \in \inte{\mathcal{Q}_4}$ with $a_1 > 0$ and $a_2 <0$.
\begin{enumerate}
  \item If $a_1 = -a_2$, we have $A_2a=(0, a_2)^{\top} \in \mathcal{Q}_3$ and $A_3a=(a_1,0)^{\top} \in \mathcal{Q}_1$.
Then there exists $\alpha >0$ and integer $k_0$ such that for any integer $l\ge k_0$, $N_l(\Sigma_3, A_2a) \le \alpha l$ and $N_{l}(\Sigma_3, A_3a) \le \alpha l$.
Thus for any integer $k \ge k_0+1$, $N_k(\Sigma_3, a) \le N_{k-1}(\Sigma_3, A_2a) + N_{k-1}(\Sigma_3, A_3a) \le \alpha(k-1) + \alpha (k-1) \le 2\alpha k$.
Therefore, $N_k(\Sigma_3, a)=O(k)$.
\item If $a_1 < -a_2$, we have $A_2a=(a_1+a_2, a_2)^{\top} \in \mathcal{Q}_3$ and $A_3a=(a_1, a_1+a_2)^{\top} \in \inte{\mathcal{Q}_4}$.
Then there exists $\alpha >0$ and integer $k_0$ such that for any integer $l\ge k_0$, $N_l(\Sigma_3, A_2a) \le \alpha l$.
For any $k \ge k_0+1$, $N_k(\Sigma_3, a) \le N_{k-1}(\Sigma_3, A_2a) + N_{k-1}(\Sigma_3, A_3a) \le \alpha(k-1) + f_{k-1}$.
Then for any $k \ge k_0+1$, $f_k \le \alpha(k-1) + f_{k-1}$.
Thus for any $k \ge 2k_0$,
\begin{align*}
f_k & \le \alpha(k-1) + f_{k-1} \le \alpha(k-1) + \alpha(k-2) + f_{k-2}\\
& \qquad \cdots \le \alpha(k-1) + \alpha(k-2) + \cdots + \alpha k_0 + f_{k_0}\\
& \le \alpha \frac{(k-1+k_0)(k-k_0)}{2} + 2^{k_0} \le \beta k^2,
\end{align*}
for some $\beta > 0$.
Therefore, $f_k=O(k^2)$.

  \item If $a_1 > -a_2$, it can be proved that $f_k=O(k^2)$ with a similar argument as in the case $a_1 < -a_2$.
	%we have $A_2a=(a_1+a_2, a_2)^{\top} \in \inte{\mathcal{Q}_4}$ and $A_3x(0)=(a_1, a_1+a_2)^{\top} \in \mathcal{Q}_1$.
%%
%Then there exists $\alpha >0$ and integer $k_0$ such that for any integer $k \ge k_0+1$,
%\begin{align*}
%N_k(\Sigma_3, a) &\le N_{k-1}(\Sigma_3, A_2a) + N_{k-1}(\Sigma_3, A_3a)\\
%&\le f_{k-1} + \alpha(k-1).
%\end{align*}
%%
%By a similar argument as in the case $a_1 < -a_2$, we can show that $f_k=O(k^2)$.
\end{enumerate}
%
%In all the cases above, we have $f_k = O(k^2)$.
%
%Therefore, $N_k(\Sigma_3, a)=O(k^2)$ for any $a \in \inte{\mathcal{Q}_4}$.
\end{proof}

\subsection{$\Sigma_4=\{A_4, A_5\}$}
%We will use the result that $N_k(\Sigma_3) = O(k^2)$ to show that $N_k(\Sigma_4)=O(k^2)$.
%
%\begin{lemma}\label{lemma_case4}
%For any even non-negative integer $k$, every product of $k$ matrices with $A_2$ and $A_3$ can be represented by a product of $k$ matrices with $A_4$ and $A_5$ and vice versa.
%\end{lemma}
%%
%\begin{proof}
    %Since $A_1 = \begin{bmatrix}
                 %0 & 1 \\
                 %1 & 0
               %\end{bmatrix}$, it can be verified directly that
    %\begin{align*}
      %A_2 & = A_4A_1,\\
      %A_3 & = A_5A_1,\\
      %A_4 &= A_1A_5A_1,\\
      %A_5 &= A_1A_4A_1.
    %\end{align*}
    %Therefore, for any product of $2$ matrices with $A_2, A_3$, it can be represented by a product of $2$ matrices with $A_4, A_5$ and vice versa:
    %\begin{align*}
      %A_2A_2 & = A_4A_1A_4A_1 = A_4A_5,\\
      %A_2A_3 & = A_4A_1A_5A_1 = A_4A_4,\\
      %A_3A_2 & = A_5A_1A_4A_1 = A_5A_5,\\
      %A_3A_3 & = A_5A_1A_5A_1 = A_5A_4.
    %\end{align*}
    %When $k$ is even, for any product of $k$ matrices with $A_2, A_3$, we can change it into products of  $A_4, A_5$ two by two and vice versa. Therefore, the lemma is proved.
  %\end{proof}

\begin{proposition} \label{prop:A4A5}
The pair $\Sigma_4$ has the \polyv{} property and $N_k(\Sigma_4) = O(k^2)$.											
\end{proposition}
\begin{proof}
First observe that $A_4A_5 = A_2A_2$, $A_4A_4 = A_2A_3$, $A_5A_5 = A_3A_2$, and $A_5A_4 = A_3A_3$.
When $k$ is an even integer, every product of $k$ matrices with $A_2$ and $A_3$ can be represented by a product of $k$ matrices with $A_4$ and $A_5$ and vice versa.
Therefore, for any given $a\in \Re^2$, $P_k(\Sigma_4, a) = P_k(\Sigma_3, a)$ and $N_k(\Sigma_4, a) = N_k(\Sigma_3, a)$.
When $k$ is an odd integer, $P_k(\Sigma_4, a) = \conv{A_4P_{k-1}(\Sigma_4,a) \cup A_5 P_{k-1}(\Sigma_4, a)}$ and $N_k(\Sigma_4, a) \le 2N_{k-1}(\Sigma_4, a) = 2N_{k-1}(\Sigma_3, a)$.
Since there exists $\alpha > 0$ and integer $k_0$ such that $N_k(\Sigma_3, a) \le \alpha k^2$ for any integer $k \ge k_0$, we have $N_k(\Sigma_4,a) \le 2 \alpha k^2$ for any integer $k \ge k_0$.
Therefore, $N_k(\Sigma_4)=O(k^2)$.
\end{proof}

\subsection{$\Sigma_5=\{A_2, A_4\}$}
\begin{proposition}\label{prop:A2A4:Q1:a1>=a2}
For any $a \in \mathcal{Q}_1$ with $a_1 \ge a_2$, $N_k(\Sigma_5, a) = O(k)$.
\end{proposition}

%Firstly, we show below that the polytope $P_k(\Sigma_5, a)$ is in the half space  $\{x\in \Re^2_+ \mid x_1 \ge  x_2\}$.
%
%\begin{lemma}\label{lemma:A2A4:Q1:x1>=x2}
%For any $a \in \mathcal{Q}_1$ and integer $k \ge 1$, the polytope $P_k(\Sigma_5, a) \subseteq \{x\in \Re^2_+ \mid x_1 \ge  x_2\}$.
%\end{lemma}
%
%\begin{lemma}\label{lemma:A2A4:Q1:E1}
\begin{proof}
First similar to the proofs of Lemmas~\ref{lemma:A1A2:Q1:Q1},~\ref{lemma:A1A2:Q1:Q3}, and~\ref{lemma:A1A2:Q1:Q24}, we can show by induction that for any $a \in \mathcal{Q}_1$ with $a_1 \ge a_2$, $E^1_{k}(\Sigma_5, k) = \{A_4^k a\}$ and $E^3_{k}(\Sigma_5, a) = \{A_2^k a\}$ when $k \ge 0$, and $|E^4_k(\Sigma_5, a)|  \leq |E^4_{k-1}(\Sigma_5, a)| + |E^1_{k-1}(\Sigma_5, a)| + 2$ and $|E^2_k(\Sigma_5, a)|  \leq |E^2_{k-1}(\Sigma_5, a)| + |E^1_{k-1}(\Sigma_5, a)| + 2$ when $k \ge 1$, respectively.
Then $|E^4_k(\Sigma_5, a)| \leq |E^4_{k-1}(\Sigma_5, a)| + 3
            % & \leq |E^4_{k-2}(\Sigma_5, a)| + 6\\
            % & \qquad \cdots\\
             \leq |E^4_{1}(\Sigma_5, a)| + 3(k-1) \le 3k-1$.
Similarly, $|E^2_k(\Sigma_5, a)| \leq 3k-1$.
Finally, for any integer $k \geq 1$, $N_k(\Sigma_5, a)  \leq \sum_{i=0}^4|E^i_k(\Sigma_5, a)| \le 6k + 8$.
\end{proof}

%%%%%%%%%%%%%%%%%%%%%%%%%%%%%%%%%%%
%The discussion here is necessary because if $v > u$, then the form of $E_k(+, +)$ will be $A_4A_2^{k-1}x(0)$.
%%%%%%%%%%%%%%%%%%%%%%%%%%%%%%%%%%%%%
\noindent We now extend Proposition~\ref{prop:A2A4:Q1:a1>=a2} to the case where $a$ is in the first quadrant.
\begin{proposition}\label{prop:A2A4:Q1}
For any $a \in \mathcal{Q}_1$, $N_k(\Sigma_5, a) = O(k)$.
\end{proposition}
\begin{proof}
%Since $P_k(\Sigma_5, a) = \conv{P_{k-1}(\Sigma_5, A_2a) \cup P_{k-1}(\Sigma_5, A_4a)}$,
%%
%we have
%\begin{align*}
    %N_k(\Sigma_5, a) \leq N_{k-1}(\Sigma_5, A_2a) + N_{k-1}(\Sigma_5, A_4a).
%\end{align*}
%
For any $a = (a_1, a_2)^{\top}$ with $a_1 \ge 0$ and $a_2 \ge 0$, both $A_2a$ and $A_4a$ are contained in the set $\{x \in \Re^2_+ \mid x_1 \ge x_2\}$.
By Proposition~\ref{prop:A2A4:Q1:a1>=a2}, there exists $\alpha >0$ and integer $k_0$ such that for any integer $l \ge k_0$, $N_l(\Sigma_5, A_2a) \le \alpha l$ and $N_l(\Sigma_5, A_4a) \le \alpha l$. 
Thus for any integer $k \ge k_0+1$,
$N_k(\Sigma_5, a) \leq N_{k-1}(\Sigma_5, A_2a) + N_{k-1}(\Sigma_5, A_4a) \le \alpha (k-1) + \alpha(k-1) \le 2\alpha k$.
Therefore, $N_k(\Sigma_5, a) = O(k)$.
\end{proof}

\noindent Finally, we extend the result to $a \in \Re^2$, similar to Proposition~\ref{prop:A2A3} for the case $\Sigma_3$.
\begin{proposition} \label{prop:A2A4}
The pair $\Sigma_5$ has the \polyv{} property and $N_k(\Sigma_5)=O(k^2)$.
% and \prob{} can be solved in $O(K^3\log K)$ time.
\end{proposition}

\section{Computational results} \label{sec:computation}
In this section, we compare the performance of our algorithm with one state-of-the-art global optimization solver Baron~\cite{kilincc2018exploiting}.
We randomly generate 10 instances for each of the 10 sets of parameters $(n,m,K)$ for \prob{}, with 100 instances in total. 
The parameters are summarized in Table~\ref{table:finalresult}.
The entries of each matrix are randomly drawn from a uniform distribution over $[-1,1]$, and the entries of the initial vector $a$ are randomly drawn from a uniform distribution over $[0,1]$.
Note that our algorithm does not rely on any additional property of $f$ other than convexity.
In order for Baron to gain a better performance, we choose a simple smooth objective function $f(x)=\|x\|_2^2$.
All test instances can be downloaded at \url{https://github.com/qqqhe}.
The mixed-integer nonlinear programming (MINLP) formulation of \prob{} is given in~\eqref{eq:minlp} and solved by Baron, where $A_{lij}$ denotes the $(i,j)$-th entry of the $l$-th matrix for $l \in [m]$.
Note that we also tried to linearize the constraints in the MINLP formulation by introducing big-M constants, but we observed that Baron and a commercial mixed-integer linear programming solver Gurobi~\cite{gurobi} easily run into numerical issues with many big-M constants in the constraints, even for a small-sized instance. 
\begin{equation} 
\label{eq:minlp}
\begin{split}
\max_{x,z} \qquad & \sum_{i=1}^nx^2_i(K) \\
\text{s.t.} \qquad & x_i(k) = \sum_{l=1}^m \sum_{j=1}^n A_{lij}x_j(k-1)z_{k,l}, i\in[n], k\in [K],\\
&\sum_{l=1}^mz_{k,l}=1, k\in[K],\\
& z_{k,l} \in \{0,1\}, l \in [m], k\in [K],\\
& x(0) = a.
\end{split}
\end{equation}

Our algorithm is coded in Matlab. 
Computational experiments are conducted on a laptop with Intel i7-6560U 2.20 GHz and 8 GB of RAM memory, under Windows 10 Operating System.
The MINLP formulation is coded in AMPL and solved by Baron 18.5.8.
The time limit for each instance is set to 600s.
When $n \leq 5$, our algorithm employs Matlab's build-in function \textsl{convhulln} to construct the set of extreme points directly. 
When $n \geq 6$, our algorithm solves a linear program with the commercial solver Gurobi~\cite{gurobi} to identify each extreme point.
%
%
%\paragraph{Results.}
The computational results are summarized in Table~\ref{table:finalresult}.
All test instances are solved to optimality by our algorithm within the time limit.
The average solution time of our algorithm is reported in the rows ``Our algorithm (s)''.
On the other hand, Baron cannot solve most instances to optimality, and has a variety of output for instances of different sizes.
Instead of reporting the solution time, we report the number of instances with different outputs by Baron in three categories that were described in~\cite{neumaier2005comparison}:
The symbol \textsf{G} (\textsf{G!}) denotes that Baron finds a global optimal solution and proves (cannot prove) its optimality within the time limit; The symbol \textsf{Limit} denotes that Baron finds some feasible solution within the time limit; The symbol \textsf{Wrong} denotes that Baron reports infeasibility or failure.

\begin{table}[htb]
\centering
\begin{tabular}{|c|c|ccccc|}
\hline
\multicolumn{2}{|c|}{$(n,m,K)$} & (2,2,20) & (2,2,50) & (2,2,500) & (2,5,500)& (2,10,500)\\
\hline
\multicolumn{2}{|c|}{Our algorithm (s)} &0.013 &0.031&0.300 & 0.298&0.289\\
\hline
\multirow{3}{*}{Baron}  &  \textsf{G}/\textsf{G!}& 4/6 &2/5 &4/2 &0/0 & 0/0\\
                        &\textsf{Limit} &0& 2&1 &7 &7\\
                       & \textsf{Wrong} &0& 1&3&3&3\\
\hline\hline
\multicolumn{2}{|c|}{$(n,m,K)$} & (5,2,100) & (5,5,100) & (5,10,100) & (8,2,50)& (10,2,20)\\
\hline
\multicolumn{2}{|c|}{Our algorithm (s)}  &1.094 &2.456 &2.405 &59.457 &58.357\\
\hline
\multirow{3}{*}{Baron}  &  \textsf{G}/\textsf{G!} & 0/0 &0/0 &0/0 &0/0 & 0/0\\
                        &\textsf{Limit} &0& 0&1 &0 &10\\
                       & \textsf{Wrong} &10& 10&9&10&0\\
\hline
\end{tabular}
\caption{The average running time of our algorithm and solution statistics of Baron}
\label{table:finalresult}
\end{table}

Our proposed algorithm has a clear advantage over Baron in solving \prob{}. 
Our algorithm is very efficient in solving instances with $n=2$ and large $m$ and $K$, requiring less than one second.
When $n$ increases to $8$ and $10$, our algorithm is able to solve instances with $K=50$ and $K=20$ respectively in less than one minute. 
On the other hand, Baron is only able to solve several instances with a pair of $2 \times 2$ matrices to optimality.
When $n$ or $m$ is larger than 2, it either cannot find the optimal solution within the time limit or runs into numerical issues.
Finally, we observe that when the problem dimension $n \ge 8$, our algorithm is not able to solve instances with $K=100$ within the time limit, since the running time grows rapidly with $K$.
We suspect the reason is that the set of randomly generated matrices no longer has the \polyv{} property for larger $n$.
This observation is also consistent with the fact that \prob{} is NP-hard for general $n$.
%

%If there are only two matrices $A$ and $B$, then the formulation is as follows:
%\begin{align}
%\max \qquad & \sum_{i=1}^nx^2_i(K) \\
%\text{s.t.} \qquad & x_i(k) =  \sum_{j=1}^n (A_{ij} - B_{ij})x_j(k-1)y_k + \sum_{j=1}^nB_{ij}x_j(k-1), i\in[n], k\in [K],\\
%& y_k \in \{0,1\}, k\in [K],\\
%& x(0) = a.
%\end{align} 

\section{Open Problems and Conclusions}
\label{sec:conclusion}
The problem \prob{} has many applications in operations research and control, and can also be seen as an approximation to the dynamics of more general continuous-time nonlinear switched systems.
In this paper, we preset an efficient exact algorithm to solve large-sized instances of \prob{} that cannot be handled by state-of-the-art optimization software.
We introduce an interesting property---the \polyv{} property---for a finite set of matrices to help analyze the time complexity of our algorithm.
We now present several open questions on the \polyv{} property, which we believe may be of independent interest.
\begin{enumerate}
	\item Does any finite set of $2 \times 2$ matrices with rational entries have the \polyv{} property?
	\item Does any finite set of $2 \times 2$ real matrices have the \polyv{} property?
	\item Is there an ``easy-to-check'' necessary condition for a set of matrices to have the \polyv{} property? Is there a finite-time algorithm to test the \polyv{} property for a given set of matrices with rational entries? If so, is deciding whether such a set of matrices has the \polyv{} property in P or NP?	
	\item Does the finiteness property imply the \polyv{} property, and vice versa?
	\item Is $N_k(\Sigma)=O(k)$ for any pair of $2 \times 2$ binary matrices?
\end{enumerate}
The last question comes from our observation that $N_k(\Sigma, a)$ grows linearly with $k$ for any $2 \times 2$ binary matrices in the computational experiment.
We believe an answer to any of the above questions will be instrumental in designing a faster exact algorithm for \prob{}.
%

%
%\noindent Based on some preliminary computational results on other $2\times 2$ matrices, we also conjecture that, as stated in Conjecture~\ref{conj:complexity:2by2matrices}, any pair of $2 \times 2$ real matrices has the \polyv{} property and \prob{} is polynomial solvable for a pair of $2 \times 2$ integer matrices.
%%
%Finally, it would be interesting and very challenging to investigate when $m$ $n \times n$ matrices has the \polyv{} property for $m \ge 2$ and $n \ge 3$.

%
%To solve \prob{}, we developed an algorithm that sequentially constructs the convex hull of all possible values of the state vector at a fixed period.
%%
%To investigate the computational complexity of the algorithm, we introduced a concept called the \polyv{} property for a pair of matrices. 
%%
%This property captures how fast the number of extreme points of the convex hull of all possible state vectors at period $k$ grows with $k$.
%%
%Some techniques we develop to prove the \polyv{} property for a pair $2 \times 2$ binary matrices can also be used in more general cases.
%  

%\begin{acknowledgements}
%\end{acknowledgements}

\bibliographystyle{plain}      % mathematics and physical sciences
\bibliography{ss}   % name your BibTeX data base

\end{document}